\documentclass[11pt,a4paper]{amsart}
\usepackage{amssymb,amsmath,amsthm}
\usepackage{extarrows}

\newcommand{\R}{\mathbb{R}}

\newcommand{\s}{\mathbb{S}}
\newcommand{\C}{\mathcal{C}}
\newcommand{\J}{\mathcal{J}}
\newcommand{\F}{\mathcal{F}}
\newcommand{\T}{\mathcal{T}}

\newtheorem{theorem}{Theorem}[section]
\newtheorem{lemma}[theorem]{Lemma}
\newtheorem{proposition}[theorem]{Proposition}
\newtheorem{corollary}[theorem]{Corollary}
\newtheorem{remark}{Remark}[section]
\theoremstyle{definition}

\numberwithin{equation}{section}

  \title{  Improved Beckner's inequality for axially symmetric functions on $\mathbb{S}^n$ }

\author{Changfeng Gui}
\address{Changfeng Gui, The University of Texas at San Antonio, TX, USA.
}
\email{changfeng.gui@utsa.edu}
\author{Yeyao Hu}
\address{Yeyao Hu, School of Mathematics and Statistics, HNP-LAMA, Central South University,
  Changsha, Hunan 410083, P. R. China}
\email{huyeyao@gmail.com }
\author{Weihong Xie}
\address{Weihong Xie, School of Mathematics and Statistics, HNP-LAMA, Central South University,
  Changsha, Hunan 410083, P. R. China}
\email{wh.xie@csu.edu.cn}

\begin{document}
 \begin{abstract}
In this article we present various uniqueness and existence results for Q-curvature type equations with a Paneitz operator on $\s^n$ in axially symmetric function spaces.  In particular, we show  uniqueness results for $n=6, 8$ and improve  the best constant of Beckner's inequality in these dimensions for axially symmetric functions under the constraint that their centers of mass are at the origin. As a consequence, the associated first Szeg\"o limit theorem is also proven for axially symmetric functions.
\end{abstract}

\date{today}

\maketitle

\vskip6mm


\vskip6mm

\vskip4mm
{\section{Introduction and statement of  results}}

 In this paper, we  consider the constant  $Q$-curvature-type  equation
    \begin{equation}\label{n-mfe}
    \alpha P_n u+(n-1)!\left(1-\frac{e^{nu}}{\int_{\s^n}e^{nu}dw}\right)=0, \quad \text{on } \mathbb S^n,
    \end{equation}
where $\mathbb S^n$ is the $n$-dimensional sphere  and
  \begin{equation*}
   P_n=\begin{cases}
       \prod_{k=0}^{\frac{n-2}{2}}(-\Delta+k(n-k-1)),&\mbox{ for $n$ even};\\
       \left(-\Delta+\left(\frac{n-1}{2}\right)^2\right)^{1/2}\prod_{k=0}^{\frac{n-3}{2}}(-\Delta+k(n-k-1)),&\mbox{ for $n$ odd}
   \end{cases}
\end{equation*}
 is the Paneitz operator on $\mathbb{S}^n$ and $\alpha$ is a positive constant.  The volume form  $dw$ is  normalized so that $\int_{\mathbb{S}^n} dw=1$.

 The corresponding functional is defined in $H^{\frac{n}{2}}(\s^n)$  by
\begin{equation}\label{functional}
J_\alpha(u)=\frac{\alpha}{2} \int_{\s^n}(P_nu) udw+(n-1)!\int_{\s^n} udw-
\frac{(n-1)!}{n}\ln \int_{\s^n}e^{nu}dw.
\end{equation}

If $\alpha =1$, \eqref{n-mfe}  corresponds to the constant $Q$-curvature equation on $\mathbb{S}^n$. It is shown in \cite{Beckner} that the following Beckner's inequality,  a higher order Moser-Trudinger type inequality, holds
\begin{equation}
\label{beckner}
J_{1} \ge 0, \quad u \in H^{\frac{n}{2}}(\s^n).
\end{equation}
Furthermore, $J_1$ is invariant under the
conformal transformation
$$ u(\xi) \to v(\tau \xi)+  \frac1n ln( |det(d\tau)(\xi)|)$$
where $\tau$ is an element of the conformal group of $\mathbb{S}^n$ and  $ |det(\cdot)|$ is the modulus of the corresponding Jacobian determinant. Equality in \eqref{beckner} is only attained at functions of the form
$$
u(\xi)= -\ln(1- \zeta \cdot \xi)+ C, \quad   C  \in \mathbb{R},
$$
where $\zeta \in B^{n+1}:=\{ \xi\in \mathbb{R}^{n+1}, |\xi| <1\}$.
(See  also  \cite{Chang95}.) In particular, \eqref{n-mfe} with $\alpha=1$ has a family of axially symmetric solutions
\begin{equation*}
  u(\xi)=-\ln(1-a \xi_1), \quad \xi\in \mathbb{S}^n \mbox{ for }a\in (-1, 1).
\end{equation*}

On the other hand,  an improved Aubin-type inequality   has been shown in \cite[Lemma 4.6]{Chang95}: for any $\alpha>1/2$, there exists a constant $C(\alpha)\ge 0$ such that  $J_{\alpha}(u) \ge -C(\alpha)$  provided that   $u$ belongs to  the  set  of  functions  with center of mass at the origin
 \begin{equation*}
   \mathfrak L=\{v\in  H^{\frac{n}{2}}(\s^n):\int_{\mathbb{S}^n}e^{nv}\xi_jdw=0;\ j=1,2,\cdots,n+1\}.
 \end{equation*}

This gives rise to the existence of a minimizer $u_0$ of $J_{\alpha} $ in $ \mathfrak L$, and $u_0$ satisfies
the corresponding Euler-Lagrange equation
 \begin{equation}\label{lagrange}
    \alpha P_n u+(n-1)!\left(1-\frac{e^{nu}}{\int_{\s^n}e^{nu}dw}\right)= \sum_{i=1}^{n+1} a_i \xi_i e^{nu} , \quad \hbox{on} \quad  \mathbb{S}^n
    \end{equation}
  for some constants $a_i, i=1, 2, \cdots n+1$. Furthermore, by exploiting the invariance of $J_1$ under the conformal transformation, \cite[Remarks (3) (ii)  for Cor. 5.4]{Chang95}
 implies  that  the following Kazdan-Warner condition
   \begin{equation}\label{kw}
    \int_{\mathbb{S}^n} \langle\nabla Q, \nabla \xi_i\rangle e^{nu}dw=0,  \quad i=1, 2, \cdots n+1
    \end{equation}
 is also applicable for    the prescribing $Q$-curvature equation
 \begin{equation*}
    P_n u+ (n-1)! -Q e^{nu}=0, \quad \xi \in   \mathbb{S}^n.
    \end{equation*}
 It is an immediate consequence that $a_i=0,  i=1, 2, \cdots n+1$  in \eqref{lagrange}. (See \cite{Wei-Xu}, proof of  Theorem 2.6.) This argument    is reminiscent of that  in \cite[Cor. 2.1]{CY87} on the constant Gaussian curvature type equation,  or the mean field equation on $ \mathbb{S}^2$,
  \begin{equation}\label{gausian}
    -\alpha \Delta u + \left(1-\frac{e^{2u}}{\int_{ \mathbb{S}^2} {e^{2u}}dw}\right)=0, \quad \xi\in  \mathbb{S}^2.
    \end{equation}
  For \eqref{gausian},  there is a vast literature. See, e.g., \cite{CY87}, \cite{GM} and references therein. Moreover,
interested reader is referred to  \cite{Chang97,DHL00,Dj08,Hang16,Hang163,Ly19,Lin98,M06,Wei99} for literature  on equations that have conformal structure.
\par
 In what follows, we shall consider axially symmetric functions that are only dependent on $\xi_1$.  The first aim of this article discusses the classification of axially symmetric solutions for \eqref{n-mfe}
  at the critical parameter $\alpha=\frac1{n+1}$. We have the following theorem.
   \begin{theorem}\label{n}
  If $\alpha=\frac{1}{n+1}$   and $u$ is an axially symmetric solution to \eqref{n-mfe} with $n\ge2$, then $u$ must be constant.
  \end{theorem}

  The rest of this paper focuses on  \eqref{n-mfe} with $n$ even in the axially symmetric setting.  We shall show that \eqref{n-mfe}  admits only constant solutions when $\alpha$ belongs to some suitable sub interval in (1/2, 1) for $n=6, 8$.   As a consequence we obtain an improved Aubin-type inequality for axially symmetric functions in $\mathfrak L$. Note that the case $n=4$ has been considered in \cite{Gui20} and similar results are obtained.

 Considering solutions axially symmetric about $\xi_1$-axis and denoting $\xi_1$ by $x$, we can reduce \eqref{n-mfe}  to
 \begin{equation}\label{n-ode}
  \alpha\left(-1\right)^\frac{n}{2}[(1-x^2)^\frac{n}{2} u']^{(n-1)}+(n-1)!- \frac{(n-1)!\sqrt\pi\Gamma\left(\frac{n}{2}\right)}{ \Gamma\left(\frac{n+1}{2}\right)\gamma}e^{nu}=0,
 \end{equation}
 where $\gamma=\int_{-1}^1(1-x^2)^\frac{n-2}{2}e^{nu}$.

One can refer to Section  3, for the detailed derivation of \eqref{n-ode}. By direct computations, we see that the corresponding functional $ I_\alpha(u)$  in $H^{\frac{n}{2}}(-1,1)$ can be expressed  as follows
 \begin{equation*}
 \begin{aligned}
     I_\alpha(u)&= \left(-1\right)^\frac{n}{2}
     \frac{\alpha}2\int_{-1}^1 (1-x^2)^\frac{n-2}{2}[(1-x^2)^\frac{n}{2} u']^{(n-1)}u+(n-1)!\int_{-1}^1 (1-x^2)^\frac{n-2}{2} u\\
     &- \frac{(n-1)!\sqrt\pi\Gamma\left(\frac{n}{2}\right)}{ n\Gamma\left(\frac{n+1}{2}\right)}\ln\left( \frac{\Gamma\left(\frac{n+1}{2}\right)}
     {\sqrt\pi\Gamma\left(\frac{n}{2}\right)}\int_{-1}^1(1-x^2)^\frac{n-2}{2}e^{nu} \right).
 \end{aligned}
\end{equation*}
where $H^{\frac{n}{2}}(-1,1)$ is defined  as the restriction of  $H^{\frac{n}{2}}(\s^n)$ in  the set of functions axially symmetric about $\xi_1$-axis and $\xi_1=x$. Moreover, the set $\mathfrak L$ is replaced  by
\begin{equation}\label{axi-constraint}
\mathfrak L_r=\bigg\{u\in H^2(\mathbb{S}^4):u=u(x) \mbox{ and } \int_{-1}^1x(1-x^2)^\frac{n-2}{2}e^{nu} =0\bigg\}.
\end{equation}

Let
\begin{equation*}
  \alpha^{(6)}=\frac{115 + \sqrt{2851}}{273} \approx 0.6168 \mbox{ and }\alpha^{(8)}=\frac{19}{23} \approx 0.8261.
\end{equation*}
Now we state the main results.
  \begin{theorem}\label{even}
 Let $n=6$ or $8$. If $\alpha^{(n)}\leq \alpha<1$, then  \eqref{n-ode} admits only constant solutions.  As an immediate consequence,  we have
$$\inf_{u\in \mathfrak L_r}I_\alpha(u)=0.
$$
 \end{theorem}

We believe that $J_{1/2}(u) \ge 0$ for $u\in \mathfrak L $, given the similar inequality for $ \mathbb{S}^2$ as shown in \cite{GM}.

Next we  define the following first momentum functionals on $H^{\frac n2}(\mathbb{S}^n)$
\begin{equation*}
\J_\alpha(u)=\frac{\alpha}{2} \int_{\s^n}(P_nu) udw+(n-1)!\int_{\s^n} udw-
\frac{(n-1)!}{2n}\ln\left( \left(\int_{\mathbb{S}^n}e^{nu}dw\right)^2 -\sum_{i=1}^{n+1} (\int_{\mathbb{S}^n}e^{nu} \xi_i dw)^2 \right).
\end{equation*}
Note that $\J_{\alpha} (u)= J_{\alpha}(u)$ when $ u \in \mathfrak L$.

As a consequence of Theorem \ref{even},  we have the following
form of first Sz\"ego limit theorem on $ \mathbb{S}^n$ for axially symmetric functions.

 \begin{theorem}\label{Szego}
Let $n=6$ or $8$, then
$$ \J_{\frac n{n+1}} (u) \ge 0, \quad \forall u \in \{ u\in H^{\frac n2}(\mathbb{S}^n):  u(\xi)=u(\xi_1)\}.
$$
 \end{theorem}

Using a bifurcation approach and Theorem \ref{n}-\ref{even}, we can also show the existence of non constant axially symmetric solution for $\alpha \in (\frac1{n+1}, \frac12).$

 \begin{theorem}\label{nontrivial}
For $n\ge 2$, there exists a non constant  solution $u_{\alpha}$ to \eqref{n-ode} for $\alpha \in (\frac1{n+1}, \frac12).$   Moreover, there exists a sequence $\alpha_m \in (\frac1{n+1},\alpha^{(n)}) $ and a sequence of non constant solutions $u_{\alpha_m}, m=1, 2, \cdots$ to \eqref{n-ode}  such that $\alpha_m \to \frac12$,  $\int_{-1}^1(1-x^2)^\frac{n-2}{2}e^{nu_{\alpha_m}} = \frac{\sqrt\pi\Gamma\left(\frac{n}{2}\right)}{\Gamma\left(\frac{n+1}{2}\right)} $ and  $\|u_{\alpha_m}\|_{L^\infty([-1,1])} \to \infty$ as $m \to \infty$.
 \end{theorem}

We also establish the following proposition concerning the centers of mass and first order momenta of solutions to \eqref{n-mfe}.
\begin{proposition}\label{n-pro}
  If $u$ solves \eqref{n-mfe}, then
  \begin{equation*}
    \int_{\mathbb{S}^n}  e^{nu}  \xi_idw=0 \quad \mbox{ and } \int_{\mathbb{S}^n}  u  \xi_idw=0, \quad i=1,2,\cdots,n+1,
  \end{equation*}
whenever $\alpha\neq1$.
\end{proposition}

The remainder of this paper is organized as follows.  First, we give some preliminaries and validate Theorem \ref{n} and Proposition \ref{n-pro} in the study of the case $n\ge 2$  in Section 2.   Section 3 is devoted to the case $n=6$ or $8$ and   the proof of Theorems \ref{even}-\ref{Szego}.
In Section 4,  we carry out a bifurcation analysis of \eqref{n-ode} and its equivalent form,  and prove Theorem \ref{nontrivial}  based on Theorems \ref{n} and \ref{even}.

{\section{Preliminaries and Classification of $\alpha=\frac{1}{n+1}$}}
\vskip 2mm
 \par
 In this section, we state several  preliminaries  which will be needed in the proof of our main results.

Note that the eigenfunctions associated with the Paneitz operator coincide with those associated with the Laplacian. It is natural to introduce Gegenbauer polynomials, see \cite[8.93]{GR}, which can be considered as a family of generalized Legendre polynomials.

 Let us first introduce the Gegenbauer polynomials (see \cite[8.93]{GR}).   Recall that
\begin{equation}\label{Ck}
  C_k^{\frac{n-1}{2}}(x)=\left(\frac{-1}{2}\right)^k\frac{  \Gamma\left(k+n-1\right)\Gamma\left(\frac{n}{2}\right)}{k! \Gamma\left(n-1\right) \Gamma\left(k+\frac{n}{2}\right)} (1-x^2)^{-\frac{n-2}{2}}\frac{d^k}{dx^k}(1-x^2)^{k+\frac{n-2}{2}}
\end{equation}
 is called Gegenbauer polynomial of order $\frac{n-1}{2}$ and degree $k$. Then $ C_k^{\frac{n-1}{2}}$ satisfies
 \begin{equation}\label{Ck-ode}
   (1-x^2)(C_k^{\frac{n-1}{2}})''-nx(C_k^{\frac{n-1}{2}})'+\bar\lambda_kC_k^{\frac{n-1}{2}}=0,\quad \ k=0,1,\cdots,
 \end{equation}
 where $\bar\lambda_k=k(k+n-1)$.

 After some calculations, it is easy to see from \cite{GR}
 that
 \begin{equation}\label{nCk'}
   |(C_k^{\frac{n-1}{2}})'|\leq  \frac{ \Gamma\left(k+n\right)}{  n\Gamma\left(n-1\right) \Gamma(k) }
 \end{equation}
 and
  \begin{equation}\label{ortho}
  \int_{-1}^1(1-x^2)^{\frac{n-2}{2}}C_k^{\frac{n-1}{2}}(x)C_s^{\frac{n-1}{2}}(x)=\begin{cases}
    \frac{\pi(k+n-2)!}{2^{(n-2)}k!(k+\frac{n-1}{2})[\Gamma(\frac{n-1}{2})]^2}
    :=A_n\frac{(k+n-2)!}{k!(k+\frac{n-1}{2})} &\quad k=s;\\
    0   &\quad k\neq s.
  \end{cases}
\end{equation}

Furthermore, we know that
 \begin{equation}\label{PnCk}
  P_n C_k^{\frac{n-1}{2}}=  \lambda_k   C_k^{\frac{n-1}{2}},
\end{equation}
where
\begin{equation}\label{lambdak}
  \lambda_k=\prod_{s=0}^{n-1}(k+s)=\frac{\Gamma(n+k)}{\Gamma(k)}.
\end{equation}
Indeed, for $n$ even,
\begin{equation*}
\begin{aligned}
    \lambda_k&=\prod_{s=0}^{\frac{n-2}{2}}[k(k+n-1) +s(n-s-1)]
=\prod_{s=0}^{\frac{n-2}{2}}(k+s)(k+n-1-s)\\
  &=\prod_{s=0}^{n-1}(k+s).
\end{aligned}
\end{equation*}
The final formula also works when $n$ is odd.

We now prove Proposition \ref{n-pro}.
Since \eqref{n-mfe} is invariant under addition by a constant, we can normalize $u$ so that
 $\int_{\mathbb{S}^4} e^{4u}dw=1$. Then, \eqref{n-mfe}  can  be written as
\begin{equation}\label{Pn-mfe}
    \alpha P_n u=(n-1)!(e^{nu}-1), \quad \xi\in \mathbb{S}^n.
    \end{equation}
 As in \cite{CY87,KW74}, we can ultiply \eqref{Pn-mfe} by $\xi_i,\ i=1,2,\cdots,n+1$ and integrate to get
  \begin{equation*}
  {  \alpha\int_{\mathbb{S}^n}(P_n u)\xi_idw=(n-1)!\int_{\mathbb{S}^n}e^{nu}\xi_idw.}
  \end{equation*}
It is easy to see from \eqref{Ck-ode} and  \eqref{PnCk} that
\begin{equation*}
-\Delta \xi_i=\bar \lambda_1 \xi_i \mbox{ and }   P_n \xi_i=\lambda_1 \xi_i,\quad i=1,2,\cdots,n+1.
\end{equation*}
  We further have
 \begin{equation*}
 { n\alpha\int_{\mathbb{S}^n} u \xi_idw= \int_{\mathbb{S}^n}e^{nu}\xi_idw.}
 \end{equation*}
On the other hand, let
\begin{equation*}
  Q=\frac{(n-1)!}{\alpha}+(n-1)!\left(1-\frac{1}{\alpha}\right) e^{-nu}.
\end{equation*}
Then \eqref{Pn-mfe} can be reduced to
\begin{equation}
  P_n u+ (n-1)! -Q e^{nu}=0
\end{equation}
 As stated  in the Introduction,     the Kazdan-Warner condition \eqref{kw} holds.  It follows from \eqref{kw} that
\begin{equation*}
  0= n!\left(\frac{1}{\alpha}-1\right)\int_{\mathbb{S}^n}\langle\nabla u, \nabla \xi_i\rangle dw=-n!\left(\frac{1}{\alpha}-1\right)\int_{\mathbb{S}^n}  u \Delta \xi_idw=n n!\left(\frac{1}{\alpha}-1\right)\int_{\mathbb{S}^n}  u  \xi_idw.
\end{equation*}
Therefore,
\begin{equation*}
 \int_{\mathbb{S}^n}  u  \xi_idw=0 \quad \mbox{ and }\quad\int_{\mathbb{S}^n}  e^{nu}  \xi_idw=0\ \quad i=1,2,\cdots,n+1
\end{equation*}
whenever  $\alpha\neq1$. Proposition \ref{n-pro} has been proven.

Throughout this paper,   we assume that $u$ is axially symmetric w.r.t. $\xi_1$-axis, i.e.,
 $u=u(\xi_1)$ for $u\in\C^\infty(\s^n)$. We may drop the subscript for simplicity to write
 \begin{equation*}
   u=u(x),\quad x\in(-1,1).
 \end{equation*}

Next, we shall prove  the uniqueness
of axially symmetric solutions when  $\alpha=\frac1{n+1}$ in \eqref{n-mfe} for all $n\ge 2$.

Let
 \begin{equation}\label{decomp-u}
   u=\sum_{k=0}^\infty a_k C_k^{\frac{n-1}{2}}(x).
 \end{equation}
 As previously  discussed,  we can get
\begin{equation}\label{decomp-Pnu}
   P_nu=\sum_{k=0}^\infty \lambda_k a_k C_k^{\frac{n-1}{2}}(x),
\end{equation}

 Now we assume that $u$  is  solution for \eqref{Pn-mfe}   and define
\begin{equation}\label{G}
G(x)=(1-x^2)u'.
\end{equation}
One has
\begin{equation}\label{decomp-G}
 G=\sum_{k=0}^\infty b_kC_k^{\frac{n-1}{2}}(x).
\end{equation}
By the recursive relations of $C_k^{\frac{n-1}{2}}(x)$(\cite[8.939]{GR})
\begin{equation*}
\begin{aligned}
   (1-x^2)(C_k^{\frac{n-1}{2}}(x))'&=2(n-1)C_{k-1}^{\frac{n-1}{2}}(x)-kxC_k^{\frac{n-1}{2}}(x)\\
&=(k+n-1)xC_k^{\frac{n-1}{2}}(x)-(k+1)C_{k+1}^{\frac{n-1}{2}}(x),
\end{aligned}
\end{equation*}
we have
\begin{equation}\label{recursive-C}
    (1-x^2)(C_k^{\frac{n-1}{2}}(x))'=\frac{(k+n-1)(k+n-2)}{2k+n-1} C_{k-1}^{\frac{n-1}{2}}(x)-
    \frac{k(k+1)}{2k+n-1}C_{k-1}^{\frac{n-1}{2}}(x),\quad \mbox{for }k\geq1.
\end{equation}
Therefore,  we see from \eqref{decomp-G} and \eqref{recursive-C} that
\begin{equation}\label{bk}
 b_k=\begin{cases}
   \frac{(k+n)(k+n-1)}{2k+n+1}-\frac{k(k-1)}{2}a_{k-1} &\mbox{for }k\geq1;\\
   \frac{n(n-1)}{n+1}a_1  &\mbox{for }k=0.
\end{cases}
\end{equation}

   Differentiate \eqref{Pn-mfe} w.r.t $x$ and multiply both sides by $(1-x^2)$ to get
   \begin{equation*}
     (1-x^2)(P_nu)'=\frac{n!}{\alpha}e^{nu}(1-x^2)u'.
   \end{equation*}
 Replacing $e^{nu}$ by $\frac{\alpha}{(n-1)!}P_n u+1$, we derive that
 \begin{equation}\label{G-eq}
    (1-x^2)(P_nu)'=n P_nu G+\frac{n!}{\alpha}G.
 \end{equation}
Inspired by Osgood,  Phillips and  Sarnak \cite{Osgood88}, we shall compare the  coefficients in front of  $C_k^{\frac{n-1}{2}}(x)$ in  both sides of \eqref{G-eq}.
It is worthy pointing out that 1-d case is solved by comparing  Fourier coefficients in \cite{Wang17}.
\par
\begin{proof}[Proof of Theorem \ref{n}]
 We first compare the coefficients before $C_0^{\frac{n-1}{2}}(x)$. By \eqref{decomp-Pnu} and \eqref{bk}, we see that
 \begin{equation}\label{decomp-Pnu'}
 \begin{aligned}
      &\quad (1-x)^2(P_nu)'= \frac{n(n-1)}{n+1} \lambda_1a_1 C_0^{\frac{n-1}{2}}(x)\\
      &+\sum_{k=1}^\infty\left[\frac{(k+n)(k+n-1)\lambda_{k+1}}{2k+n+1}
      a_{k+1}-
   \frac{k(k-1)\lambda_{k-1}}{2k+n-3}a_{k-1}\right]C_k^{\frac{n-1}{2}}(x).
 \end{aligned}
\end{equation}
On the other hand,
 multiplying \eqref{G-eq} by $(1-x^2)^{\frac{n-2}{2}}$ and integrating, we obtain
\begin{equation*}
   \frac{2n (n-2)!n!}{n+1} A_n a_1 =n\int_{\s^n}P_nuG+\frac{2n (n-2)!n!}{n+1} A_na_1.
\end{equation*}
Equivalently, we have
\begin{equation}\label{coeffi-C0}
   \frac{2 (n-2)!n!}{n+1}\left(1-\frac{1}{\alpha}\right)a_1=\int_{\s^n}P_n uG.
\end{equation}
It remains to compute $\int_{\s^n}P_n uG$.
It follows from \eqref{decomp-Pnu}, \eqref{decomp-G}, \eqref{bk}  and \eqref{ortho} that
\begin{equation*}
  \begin{aligned}
   &\quad\int_{\s^n}P_nuG =\int_{-1}^1
   (1-x^2)^{\frac{n-2}{2}}\left(\sum_{k=0}^\infty \lambda_k a_kC_k^{\frac{n-1}{2}}(x)\right)
   \sum_{k=0}^\infty b_k C_k^{\frac{n-1}{2}}(x)\\
   &=A_n \sum_{k=1}^\infty \frac{(k+n-2)!}{k!(k+\frac{n-1}{2})}\lambda_k a_kb_k\\
   &=2A_n \sum_{k=1}^\infty\lambda_k  \left[\frac{(k+n)!}{k!(2k+n-1)(2k+n+1)}a_{k+1}a_k -\frac{(k-1)(k+n-2)!}{(k-1)!(2k+n-3)(2k+n-1)}a_{k-1}a_k \right]\\
   &=2A_n \sum_{k=1}^\infty \left[\frac{(k+n-1)!\lambda_{k+1}}{(k-1)!(2k+n-1)(2k+n+1)}a_{k+1}a_k -\frac{(k-1)(k+n-2)!\lambda_k}{(k-1)!(2k+n-3)(2k+n-1)}a_{k-1}a_k \right]\\
   &=0.
  \end{aligned}
\end{equation*}
By \eqref{coeffi-C0}, we conclude  that
\begin{equation}\label{b0}
  \mbox{ if } \alpha\neq1, \mbox{  then }a_1=0\mbox{  and so }b_0=0.
\end{equation}
\par
Then we compare the coefficients in front of $C_1^{\frac{n-1}{2}}(x)$ in \eqref{G-eq}.  More precisely,
\begin{equation}\label{coeffi-C1}
     \int_{-1}^1(1-x^2)^{\frac n2} (P_nu)'C_1^{\frac{n-1}{2}}=n \int_{-1}^1 (1-x^2)^{\frac{n-2}{2}}(P_nu)' P_nu G C_1^{\frac{n-1}{2}}+\frac{n!}{\alpha}\int_{-1}^1 (1-x^2)^{\frac{n-2}{2}}GC_1^{\frac{n-1}{2}}.
\end{equation}
From \eqref{decomp-Pnu'},
 we deduce that
\begin{equation}\label{coeffi-C1-1}
\begin{aligned}
    &\quad \int_{-1}^1(1-x^2)^{\frac n2}(P_nu)'C_1^{\frac{n-1}{2}}(x)= \frac{2 n!(n+1)! }{n+3}
      A_na_{2}.
\end{aligned}
\end{equation}
For the second term of RHS of \eqref{coeffi-C1}, we have
\begin{equation}\label{coeffi-C1-2}
\frac{n!}{\alpha}\int_{-1}^1 (1-x^2)^{\frac{n-2}{2}}GC_1^{\frac{n-1}{2}}=\frac{2 (n!)^2 }{\alpha(n+3)}
      A_na_{2}.
\end{equation}
For the first term of RHS of \eqref{coeffi-C1}, after integration by part,
we obtain
\begin{equation*}
  \begin{aligned}
    &\quad \int_{-1}^1(1-x^2)^{\frac{n-2}{2}}P_nu GC_1^{\frac{n-1}{2}}(x)=-\frac{n-1}{n}\int_{-1}^1P_nu Gd((1-x^2)^{\frac n2})\\
    &=\frac{n-1}{n}\int_{-1}^1(1-x^2)^{\frac{n-2}{2}}\left[(1-x^2)(P_nu)' G+(1-x^2)G'P_nu\right]dx\\
    &:=\frac{n-1}{n}(I+II).
  \end{aligned}
\end{equation*}
By \eqref{decomp-u}-\eqref{recursive-C}, we find
\begin{equation}\label{decomp-G'}
\begin{aligned}
    (1-x^2)G'&=\sum_{k=1}^\infty b_k\left[\frac{(k+n-1)(k+n-2)}{2k+n-1} C_{k-1}^{\frac{n-1}{2}}(x)-
    \frac{k(k+1)}{2k+n-1}C_{k-1}^{\frac{n-1}{2}}(x)\right]\\
    &=\frac{n(n-1)}{n+1} b_1C_0^{\frac{n-1}{2}}(x)+\sum_{k=1}^\infty\left(\frac{(k+n)(k+n-1)}{2k+n-1}b_{k+1}
    -\frac{k(k-1)}{2k+n-3} b_{k-1}\right)C_{k}^{\frac{n-1}{2}}(x).
\end{aligned}
\end{equation}
After some computations,  we deduce from \eqref{decomp-Pnu}, \eqref{ortho},\eqref{decomp-Pnu'} and \eqref{decomp-G'},
\begin{equation*}
\begin{aligned}
I+II&=  \frac{2(n+1)!n!}{n+3} A_n a_2b_1-\frac{2n!n!}{n+3} A_n a_2b_1
\\
&+2A_n\sum_{k=2}^\infty a_{k+1} b_k\lambda_{k+1}\left[\frac{(k+n)!}{(2k+n+1)(2k+n-1)k!}
-\frac{(k+n-1)!}{(2k+n-1)(2k+n+1)k!}\right]\\
 &+2A_n\sum_{k=2}^\infty a_k b_{k+1}\lambda_{k}\left[\frac{(k+n)!}{(2k+n+1)(2k+n-1)k!}
-\frac{(k+n-1)!}{(2k+n-1)(2k+n+1)(k-1)!}\right]\\
&=\frac{2n(n!)^2}{n+3} A_n a_2b_1+2n A_n\sum_{k=2}^\infty \frac{(k+n-1)!\lambda_{k}}{(2k+n+1)(2k+n-1)k!}\left(
\frac{k+n}{k}a_{k+1} b_k+a_k b_{k+1}\right)\\
&= \frac{2n(n-1)n!(n+1)!}{(n+3)(n+5)} A_n a_2^2+2n (n-1)A_n\sum_{k=2}^\infty \frac{\lambda_{k}(k+n)(k+n)!}{(2k+n+1)(2k+n-1)(2k+n-3)kk!}a_{k+1}^2\\
&-2n (n-1)A_n\sum_{k=2}^\infty \frac{\lambda_{k}(k+n+1)!}{(2k+n+1)(2k+n-1)(2k+n-3)(k+1)!}a_ka_{k+2}\\
&=2n (n-1)A_n\sum_{k=2}^\infty \frac{(k+n-1)!}{(2k+n+1)(2k+n-1)(k-1)!}\\
&\times\left(\frac{(k+n-1)\lambda_{k-1}}{(2k+n-3)(k-1)}a_k^2
-\frac{(k+n+1)(k+n)\lambda_{k}}{k(2k+n-3)(k-1)}a_ka_{k+2}\right)\\
&=n (n-1)A_n\sum_{k=0}^\infty \frac{(k+n-1)!\prod_{s=1}^{n-1}(k+s)}{(k+1)(2k+n+1)(k+1)!}b_{k+1}^2.
\end{aligned}
\end{equation*}
Therefore,
\begin{equation}\label{coeffi-C1-3}
\begin{aligned}
\int_{-1}^1(1-x^2)^{\frac{n-2}{2}}P_nu=(n-1)^2A_n\sum_{k=0}^\infty \frac{(k+n-1)!\prod_{s=1}^{n-1}(k+s)}{(k+1)(2k+n+1)(k+1)!}b_{k+1}^2.
  \end{aligned}
\end{equation}
Combining \eqref{coeffi-C1-1}-\eqref{coeffi-C1-3}, we have
\begin{equation*}
 \frac{2 (n!)^2 }{(n-1)^2(n+3)}
      \left(n+1-\frac{1}{\alpha}\right)a_2=\sum_{k=0}^\infty
      \frac{(k+n-1)!\prod_{s=1}^{n-1}(k+s)}{(k+1)(2k+n+1)(k+1)!}b_{k+1}^2.
\end{equation*}
It is easy to see that if $\alpha= \frac1{n+1}$ then $b_k=0$ for $k\ge 1$. Thus, $G\equiv0$ and $u\equiv C$.
\end{proof}

\vskip4mm

{\section{  The case: $n$ is even }}

In this section, we shall show some results  for \eqref{n-mfe} with $n\ge6$ even, which can be regarded as the generalization of  recent results in \cite{Gui20} for $n=4$.

Let $\theta_i,\ i=1,2,\dots,n$ denote the usual angular coordinates on the sphere with
  \begin{equation*}
    \theta_n\in[0,2\pi] \quad \mbox{ and }\theta_i\in[0, \pi],\ i=1,2,\dots,n-1
  \end{equation*}
 and define $x=\cos \theta_1$. Then the metric tensor is
 \begin{equation}\label{metirc}
 g_{ij}=
\left( \begin{matrix}
  (1-x^2)^{-1}  & 0 & 0 &\cdots & 0 \\
  0  & 1-x^2  & 0 &\cdots & 0 \\
    0 & 0 & (1-x^2)\sin^2\theta_2 &\cdots & 0 \\
      \vdots &\vdots  & \vdots &\ddots & \vdots \\
     0 & 0 & 0 &\cdots & (1-x^2)\sin^2\theta_2\dots\sin^2\theta_{n-1}
\end{matrix}
\right )
 \end{equation}
 In what follows, we shall consider axially symmetric functions which only depend  on $x$. For such functions, we have
 \begin{equation}
     \int_{S^n}dw=\frac{\Gamma\left(\frac{n+1}{2}\right)}{2\pi^{\frac{n+1}{2}}}
     \int_{-1}^1\int_0^{\pi}\int_0^{\pi}\cdots\int_0^{2\pi} (1-x^2)^\frac{n-2}{2}\sin^{(n-2)}\theta_2\cdots
     \cdots \sin\theta_{n-1} d\theta_{n}\cdots d\theta_2dx.
 \end{equation}
 Note that
 \begin{equation*}
   \int_{0}^{\frac{\pi}{2}}sin^{s-1}\theta d\theta=2^{s-2}B\left(\frac{s}{2},\frac{s}{2}\right)
   =2^{-1}B\left(\frac{1}{2},\frac{s}{2}\right).
 \end{equation*}
 Then for $k=2,3,\dots,n-1$,
 \begin{equation*}
  \int_{0}^{\pi}sin^{n-k}\theta_k d\theta_k=  B\left(\frac{1}{2},\frac{n-k+1}{2}\right)
  =\frac{\sqrt\pi
  \Gamma\left(\frac{n+1-k}{2}\right)}{\Gamma\left(\frac{n+2-k}{2}\right)}.
 \end{equation*}
 We further have
 \begin{equation*}
     \begin{aligned}
     \int_{S^n}dw&=\frac{\Gamma\left(\frac{n+1}{2}\right)}{\pi^{\frac{n-1}{2}}}\prod_{k=2}^{n-1}\frac{\sqrt\pi
  \Gamma\left(\frac{n+1-k}{2}\right)}{\Gamma\left(\frac{n+2-k}{2}\right)}
     \int_{-1}^1 (1-x^2)^\frac{n-2}{2}dx\\
     &=\frac{\Gamma\left(\frac{n+1}{2}\right)}{\sqrt\pi\Gamma\left(\frac{n}{2}\right)}\int_{-1}^1 (1-x^2)^\frac{n-2}{2}dx\\
     &=\frac{(n-1)!}{2^{n-1}[(n/2-1)!]^2}\int_{-1}^1 (1-x^2)^\frac{n-2}{2}dx.
   \end{aligned}
 \end{equation*}
 for even $n$.
Moreover,
 \begin{equation*}
    \begin{aligned}
      \Delta u&= |g|^{-\frac{1}{2}}\frac{\partial}{\partial x}\left(|g|^{\frac{1}{2}}g^{11}\frac{\partial u}{\partial x}\right)
     =(1-x^2)^{-\frac{n-2}{2}}\frac{\partial}{\partial x}\left[(1-x^2)^\frac{n}{2}\frac{\partial u}{\partial x}\right]\\
     &=(1-x^2)u''-nxu'
    \end{aligned}
 \end{equation*}
 and
\begin{equation}\label{Pn-axial}
  P_nu=\left(-1\right)^\frac{n}{2}[(1-x^2)^\frac{n}{2} u']^{(n-1)}=
  \left(-1\right)^\frac{n}{2}[(1-x^2)^\frac{n-2}{2} \Delta u]^{(n-2)}
\end{equation}
for $u=u(x)$.  Hence, we can transform the original  equation \eqref{n-mfe} on $\mathbb{S}^n$ into an ODE \eqref{n-ode}.

 In the following, we assume that $\alpha<1.$

 Let $G$ be defined as \eqref{G}.
 In view of equation \eqref{n-ode},   we drive
   \begin{equation}\label{G-eq0}
     \alpha(-1)^\frac{n}{2}((1-x^2)^\frac{n-2}{2}G)^{(n-1)}+(n-1)!
     -  \frac{(n-1)!\sqrt\pi\Gamma\left(\frac{n}{2}\right)}{ \Gamma\left(\frac{n+1}{2}\right)\gamma} e^{nu}=0.
   \end{equation}
   By differentiating \eqref{G-eq0}, we further have
   \begin{equation}\label{G-equation}
      (-1)^\frac{n}{2}(1-x^2)^\frac{n}{2} [(1-x^2)^\frac{n-2}{2}G]^{(n)} -\frac{n!}{\alpha} (1-x^2)^\frac{n-2}{2}G-
     (-1)^\frac{n}{2}n(1-x^2)^\frac{n-2}{2}G [(1-x^2)^\frac{n-2}{2}G]^{(n-1)}=0.
   \end{equation}

For simplicity, let
 \begin{equation}\label{Ck-defin}
 \hat C_k^{\frac{n-1}{2}} =\frac{ k! \Gamma\left(n-1\right) }{\Gamma\left(k+n-1\right) }C_k^{\frac{n-1}{2}} \quad   \mbox{ and }\quad d_k=\frac{\Gamma\left(k+n-1\right)}{ k! \Gamma\left(n-1\right) } b_k
 \end{equation}
 and drop the hat and the index ${\frac{n-1}{2}}$ in the notation in later discussion.
 In the following, we will study the above $ C_k $ and $d_k$.

Combining \eqref{decomp-G} and \eqref{b0},
we have the following decomposition using the orthogonal polynomials $C_k, k\ge 1$:
  \begin{equation}\label{decomp-G-1}
 G=\sum_{k=1}^\infty d_k C_k (x).
\end{equation}
Recall that
\begin{equation*}
  \bar\lambda_k=k(k+n-1) \mbox{ and }\lambda_k=\frac{\Gamma(n+k)}{\Gamma(k)}.
\end{equation*}
Note that \eqref{nCk'} and \eqref{ortho} respectively become
\begin{equation}\label{Ck'}
   |C_k'(x)| \leq \frac{\bar\lambda_k}{n},\quad  \forall x\in(-1,1);
 \end{equation}
 and
 \begin{equation}\label{Ck-ortho}
 \int_{-1}^1 (1-x^2)^\frac{n-2}{2} C_k C_l
=\frac{2^{n-1}[(n/2-1)!]^2\bar\lambda_k}{(n+2k-1) \lambda_k}\delta_{kl}.
 \end{equation}

Define $d_1=\beta$
and
\begin{equation*}
t_k^2=\begin{cases}
  d_k^2\int_{-1}^1(1-x^2)^\frac{n-2}{2}C_k^2,\quad &\mbox{ for } k\geq2;\\
 \beta^2\int_{-1}^1(1-x^2)^\frac{n-2}{2}C_1^2,\quad &\mbox{ for } k=1.
\end{cases}
\end{equation*}
It follows from \eqref{Ck-ode} and \eqref{Pn-axial} that
\begin{equation*}
 \left[(1-x^2)^\frac{n}{2} C_k'\right]'=-\bar\lambda_k(1-x^2)^\frac{n-2}{2} C_k
\end{equation*}
and
\begin{equation*}
   [(1-x^2)^\frac{n-2}{2} C_k]^{(n-2)}=(-1)^\frac{n-2}{2}\frac{ \lambda_k}{\bar\lambda_k}C_k.
\end{equation*}
After  direct calculations,  we obtain the following decompositions.

\begin{equation}\label{decomp_G^2}
  \int_{-1}^1(1-x^2)^\frac{n-2}{2} G^2=  \sum_{k=1}^\infty  t_k^2,
\end{equation}

\begin{equation}\label{decomp_G'}
  \int_{-1}^1 (1-x^2)^\frac n2(G')^2= \sum_{k=1}^\infty \bar\lambda_k t_k^2,
\end{equation}
\begin{equation}\label{decomp_Gn-2}
 \int_{-1}^1\left|[(1-x^2)^\frac{n-2}{2} G]^{(\frac{n-2}{2})}\right|^2 =
 \sum_{k=1}^\infty\frac{ \lambda_k}{\bar\lambda_k} t_k^2= \sum_{k=1}^\infty\frac{\Gamma(n+k-1)}{\Gamma(k+1)} t_k^2.
\end{equation}
\par

Next, we shall state some important integral identities which will be used frequently in the proof of the main results.

\begin{lemma}\label{equality}
We establish the following equalities for $G$.
\begin{equation}\label{C1G}
  \int_{-1}^1(1-x^2)^\frac{n-2}2  C_1 G=\frac{\sqrt\pi\Gamma\left(\frac{n}{2}\right)}{2\Gamma\left(\frac{n+3}{2}\right)}\beta,
\end{equation}
 \begin{equation}\label{enu}
 \int_{-1}^1 (1-x^2)^\frac{n}2 \frac{e^{nu}}{\gamma}=\frac{n}{n+1}(1-\alpha\beta),
 \end{equation}
\begin{equation}\label{CkG}
  \int_{-1}^1 (1-x^2)^\frac{n-2}2    C_k G=-\frac{2^{n-1}[(n/2-1)!]^2 }{\alpha\lambda_k}\int_{-1}^1 \frac{e^{nu}}{\gamma}(1-x^2)^\frac{n}2 C_k',\quad k\geq 2,
\end{equation}
\begin{equation}\label{key-id}
   \int_{-1}^1\left|[(1-x^2)^\frac{n-2}{2} G]^{(\frac{n-2}{2})}\right|^2   =\frac{\sqrt\pi(n-2)!\Gamma\left(\frac{n}{2}\right)}{\Gamma\left(\frac{n+3}{2}\right)}
   \left(n+1-\frac{1}{\alpha}\right)\beta.
\end{equation}
\end{lemma}

\begin{proof}
  A direct calculation shows that
\begin{equation*}
    \int_{-1}^1(1-x^2)^\frac{n-2}2  C_1 G= \beta  \int_{-1}^1(1-x^2)^\frac{n-2}2x^2 =\frac{\sqrt\pi\Gamma\left(\frac{n}{2}\right)}{2\Gamma\left(\frac{n+3}{2}\right)}\beta.
\end{equation*}
Then \eqref{C1G} follows.

To prove \eqref{enu} and \eqref{CkG}, multiplying  \eqref{G-eq0}
by $\int_{-1}^x (1-s^2)^\frac{n-2}{2}C_k(s)ds$ with $k\geq 1$
   and integrating over $[-1, 1]$, we have
   \begin{equation*}
   \begin{aligned}
       &\quad \int_{-1}^1\int_{-1}^x (1-s^2)^\frac{n-2}{2}C_k(s)\bigg[     \alpha(-1)^\frac{n}{2}((1-x^2)^\frac{n-2}{2}G)^{(n-1)}+(n-1)!-\frac{2^{n-1}[(n/2-1)!]^2}{\gamma}e^{nu}
       \bigg] =0.
   \end{aligned}
  \end{equation*}
After integrating by parts,     we obtain
\begin{equation}\label{22-1}
\begin{aligned}
 &\quad (-1)^\frac{n}{2} \alpha \int_{-1}^1\int_{-1}^x (1-s^2)^\frac{n-2}{2}C_k(s)((1-x^2)^\frac{n-2}{2}G)^{(n-1)}\\
    & = (-1)^\frac{n+2}{2} \alpha \int_{-1}^1((1-x^2)^\frac{n-2}{2}G)^{(n-2)}(1-x^2)^\frac{n-2}{2}C_k(x)\\
  &= \frac{\alpha\Gamma(n+k-1)}{\Gamma(k+1)}\int_{-1}^1 (1-x^2)^\frac{n-2}{2}C_k(x)G.
\end{aligned}
\end{equation}
Furthermore,
 \begin{equation}\label{22-2}
   \begin{aligned}
    \int_{-1}^1\int_{-1}^x (1-s^2)^\frac{n-2}{2}C_k(s) &= \left(x\int_{-1}^x(1-s^2)^\frac{n-2}{2}C_k(s) \right)\Big|_{-1}^1- \int_{-1}^1 (1-x^2)^\frac{n-2}{2}xC_k  \\
    &=-\frac{\sqrt\pi\Gamma\left(\frac{n}{2}\right)}{2\Gamma\left(\frac{n+3}{2}\right)}\delta_{1k}.
   \end{aligned}
 \end{equation}
By \eqref{Ck-ode}  we find that
\begin{equation}\label{22-3}
     \int_{-1}^x  (1-s^2)^\frac{n-2}{2}C_k(s)=- \frac1{\bar\lambda_k}(1-x^2)^\frac{n}{2} C_k'(x).
\end{equation}
Let $k=1$, then from \eqref{22-1}--\eqref{22-3}  we   deduce that
\begin{equation*}
 \frac{\sqrt\pi(n-1)!\Gamma\left(\frac{n}{2}\right)}{\bar\lambda_1\Gamma\left(\frac{n+1}{2}\right)} \int_{-1}^1 (1-x^2)^\frac{n}2 \frac{e^{nu}}{\gamma} =\frac{\sqrt\pi(n-1)!\Gamma\left(\frac{n}{2}\right)}{2\Gamma\left(\frac{n+3}{2}\right)} (1-\alpha\beta).
\end{equation*}
This leads to \eqref{enu}.

When $k\ge2$,  \eqref{CkG} follows from \eqref{22-1}--\eqref{22-3}. Here we employ the fact that
\begin{equation*}
  \frac{ \Gamma(n+k-1)}{\Gamma(k+1)}=\frac{\lambda_k}{\bar\lambda_k} \mbox{ and }2^{n-1}[(n/2-1)!]^2=\frac{(n-1)!\sqrt\pi\Gamma\left(\frac{n}{2}\right)}{ \Gamma\left(\frac{n+1}{2}\right)}.
\end{equation*}
\par
 For \eqref{key-id}, multiplying   \eqref{G-equation} by $x$ and integrating from $-1$ to $1$, we obtain
 \begin{equation}\label{KI-1}
 \begin{aligned}
    &\quad \int_{-1}^1  \bigg[(-1)^\frac{n}{2}x(1-x^2)^\frac{n}{2} [(1-x^2)^\frac{n-2}{2}G]^{(n)}
     -\frac{n!}{\alpha} x(1-x^2)^\frac{n-2}{2}G\\
     &-(-1)^\frac{n}{2}n x(1-x^2)^\frac{n-2}{2}G [(1-x^2)^\frac{n-2}{2}G]^{(n-1)}\bigg]=0.
 \end{aligned}
 \end{equation}
 By integrating by parts,  we  find
 \begin{equation}\label{KI-2}
 \begin{aligned}
      \int_{-1}^1(-1)^\frac{n}{2}x(1-x^2)^\frac{n}{2} [(1-x^2)^\frac{n-2}{2}G]^{(n)}&= \int_{-1}^1 (-1)^\frac{n}{2}[x(1-x^2)^\frac{n}{2}]^{(n)} [(1-x^2)^\frac{n-2}{2}G] \\
      &= (n+1)!\int_{-1}^1 x(1-x^2)^\frac{n-2}{2}G.\\
 \end{aligned}
\end{equation}
Furthermore,
\begin{equation}\label{G(n/2-1)}
\begin{aligned}
   &\quad (-1)^\frac{n}{2}n\int_{-1}^1[x(1-x^2)^\frac{n-2}{2}G][(1-x^2)^\frac{n-2}{2} G]^{(n-1)}\\
   &=n\int_{-1}^1[(1-x^2)^\frac{n-2}{2} G]^{(\frac{n-2}2)}\left[x((1-x^2)^\frac{n-2}{2}G)^{(\frac{n}2)}+\frac n2((1-x^2)^\frac{n-2}{2} G)^{(\frac{n-2}2)}\right]\\
   &=\frac{n(n-1)}{2}\int_{-1}^1|[(1-x^2)^\frac{n-2}{2} G]^{(\frac{n}{2}-1)}|^2.
\end{aligned}
\end{equation}
It follows from  \eqref{KI-1}--\eqref{G(n/2-1)}   that
\begin{equation*}
   n!\left(n+1-\frac{1}{\alpha}\right)\int_{-1}^1 x(1-x^2)^\frac{n-2}{2}G=\frac{n(n-1)}{2}\int_{-1}^1|[(1-x^2)^\frac{n-2}{2} G]^{(\frac{n}{2}-1)}|^2,
\end{equation*}
which,  joint with  \eqref{C1G},  implies \eqref{key-id}.
\end{proof}

Multiplying \eqref{G-equation} by $G$ and integrating over$[-1,1]$, we have
   \begin{equation}\label{key-equation}
     \begin{aligned}
       \int_{-1}^1&(-1)^\frac{n}{2}(1-x^2)^\frac{n}{2}G [(1-x^2)^\frac{n-2}{2}G]^{(n)} -\frac{n!}{\alpha} \int_{-1}^1(1-x^2)^\frac{n-2}{2}G^2\\
       &- (-1)^\frac{n}{2}n \int_{-1}^1(1-x^2)^\frac{n-2}{2}G^2 [(1-x^2)^\frac{n-2}{2}G]^{(n-1)}=0.
     \end{aligned}
   \end{equation}

For the first term,
\begin{equation*}
     \begin{aligned}
      \int_{-1}^1&(-1)^\frac{n}{2}(1-x^2)^\frac{n}{2}G [(1-x^2)^\frac{n-2}{2}G]^{(n)}\\
      &=(-1)^\frac{n}{2}  \int_{-1}^1[(1-x^2)^\frac{n-2}{2}G][(1-x^2)^\frac{n}{2} G']^{(n-1)}+(-1)^\frac{n}{2}n\int_{-1}^1[x(1-x^2)^\frac{n-2}{2}G][(1-x^2)^\frac{n-2}{2} G]^{(n-1)}\\
      &= \lfloor G\rfloor^2+\frac{n(n-1)}{2}\int_{-1}^1|[(1-x^2)^\frac{n-2}{2} G]^{(\frac{n}{2}-1)}|^2,
     \end{aligned}
   \end{equation*}
where
\begin{equation}\label{norm}
  \lfloor G\rfloor^2=  \frac{2^{n-1}[(n/2-1)!]^2}{(n-1)!}\int_{\mathbb{S}^n} (P_n G ) G dw= \left(-1\right)^\frac{n}{2}  \int_{-1}^1 (1-x^2)^\frac{n-2}{2}[(1-x^2)^\frac{n}{2} G']^{(n-1)} G.
\end{equation}

However,  the last term is very sophisticated, and we will consider it for the cases $n=6$ and $n=8$ below.
More precisely,   after some  complicated computations, we derive that for $n=6$,
\begin{equation}\label{n=6}
 \int_{-1}^1(1-x^2)^2 G^2 [(1-x^2)^2 G]^{(5)}= -5\int_{-1}^1(1-x^2)^4 G'(G'')^2-\frac{80}{3}
   \int_{-1}^1(1-x^2)^3(G')^3,
\end{equation}
and
for $n=8$,
\begin{equation}\label{n=8}
\begin{aligned}
   &\quad\int_{-1}^1(1-x^2)^3 G^2 [(1-x^2)^3 G]^{(7)}=1260
   \int_{-1}^1(1-x^2)^4(G')^3+252\int_{-1}^1(1-x^2)^5 G'(G'')^2\\
   &+7\int_{-1}^1(1-x^2)^6G'(G^{(3)})^2 -21\int_{-1}^1[(1-x^2)G]^{(3)}(1-x^2)^5 (G'')^2.
\end{aligned}
\end{equation}
The details of the computation are postponed to Appendix A and B.

In view of  the above relations,  \eqref{key-equation} for $n=6$ becomes
  \begin{equation}\label{n=6key-eq1}
     \begin{aligned}
      &\quad\lfloor G\rfloor^2+15 \int_{-1}^1|[(1-x^2)^2 G]^{''}|^2-\frac{6!}{\alpha} \int_{-1}^1(1-x^2)^2 G^2\\
      & -30\int_{-1}^1(1-x^2)^4 G'(G'')^2-160
   \int_{-1}^1(1-x^2)^3(G')^3=0.
     \end{aligned}
   \end{equation}
Similarly,
 \eqref{key-equation} for $n=8$  is equivalent to
  \begin{equation}\label{n=8key-eq1}
     \begin{aligned}
      &\quad\lfloor G\rfloor^2+28\int_{-1}^1|[(1-x^2)^3 G]^{(3)}|^2-\frac{8!}{\alpha} \int_{-1}^1(1-x^2)^3 G^2\\
      & -56\bigg[180
   \int_{-1}^1(1-x^2)^4(G')^3+36\int_{-1}^1(1-x^2)^5 G'(G'')^2\\
   &+\int_{-1}^1(1-x^2)^6G'(G^{(3)})^2 -3\int_{-1}^1[(1-x^2)G]^{(3)}(1-x^2)^5 (G'')^2\bigg]=0.
     \end{aligned}
   \end{equation}

  Next we shall show a family of important  gradient estimates in $\s^n$,  which generalizes a  similar  result in $\s^4$.
\begin{lemma}\label{Gj-estimate-lemma}
 For all $x\in[-1,1]$, we have
\begin{equation}\label{gradient-n}
G_j:= (-1)^j[(1-x^2)^j G]^{(2j+1)}\le \frac{(2j+1)!}{\alpha},\quad x \in (-1, 1), \quad 0\le j\le \frac n2-1.
\end{equation}
\end{lemma}
\begin{proof}

We will first prove the result for the case $n=6$.

Since
\begin{equation*}
\begin{aligned}
 \left[(1-x^2)^2  G\right]^{(5)}&=
 \left[(1-x^2)((1-x^2)G)''-4x((1-x^2)G)'-2(1-x^2)G\right]^{(3)}\\
    &=\left[(1-x^2)((1-x^2)G)^{(3)}-6x((1-x^2)G)''-6((1-x^2)G)'\right]''\\
    &=-(1-x^2)G_1''+10x G_1'+20G_1.
\end{aligned}
\end{equation*}
Therefore, by \eqref{G-eq0} we have
\begin{equation}\label{G1''-upbd}
 -(1-x^2)G_1''+10x G_1'+20G_1\leq\frac{5!}{\alpha}.
\end{equation}
Let $M_1=\max_{x\in[-1,1]}G_1'(x)$.
\par
Case 1:
\begin{equation*}
  M_1=\lim_{x_k\rightarrow 1}G_1'(x_k) \quad \mbox{ for  some } x_k\in(-1,1).
\end{equation*}
As in \cite{Gui20},  let $r=|x'|=\sqrt{1-x^2}$,  then we write
\begin{equation*}
  G(x)=\bar G(r),\ \ G_1(x)=\bar G_1(r) \mbox{ and }u(x)=\bar u(r) \quad \mbox{ for } r\in[0,1)\mbox{ and }x\in(0,1],
\end{equation*}
and $\bar u(r)$  can  be extended evenly  so that $\bar u(r)\in \mathcal{C}^\infty(-1,1)$.
Hence,
\begin{equation}
 \begin{aligned}
   G_1(x)&=\bar G_1(r)=\frac{d^3}{dx^3}(-r^2\bar G(r)) \\
   &=\frac{d^3}{dx^3}[r^3\sqrt{1-r^2}\bar u_r]
   :=\frac{d^3}{dx^3}[g(r^2)],
    \end{aligned}
\end{equation}
where $g(s),\ s\in(-1,1)$ is a $\C^\infty$ function.
\par
Let $g_0(s)=g(s)$ and
\begin{equation*}
  g_i(s)=-2g_{i-1}'(s)\sqrt{1-s},\quad i=1,2,3.
\end{equation*}
 Then $g_i(s)\in \C^\infty(-1,1)$ for $i=1,2,3$.   Differentiate $g(r^2)$ with respect to $x$:
 \begin{equation*}
 \begin{aligned}
     \frac{d}{dx}g(r^2)=-2g'(r^2)\sqrt{1-r^2}=g_1(r^2).
 \end{aligned}
 \end{equation*}
Similarly, we have
\begin{equation*}
  \frac{d^i}{dx^i}g(r^2)=g_i(r^2),\quad i=2,3.
\end{equation*}
Therefore, $G_1(x) =\bar G_1(r)\in\C^\infty(-1,1)$ and is even.
We now can write
\begin{equation}\label{G1(r)}
 \bar G_1(r)=c_1+c_2r^2+c_3r^4=O(r^6) \quad \mbox{near }r=0.
\end{equation}
     Then $c_2\leq0$, since $M_1=G_1(1)=\bar G_1(0)$.  Note that  near $r=0$,
\begin{equation*}
  \begin{aligned}
   x  G_1'(x)&=\sqrt{1-r^2}\frac{d\bar G_1'(x)}{dr}\frac{dr}{dx}\\
    &=-2c_2+O(r^2)
  \end{aligned}
\end{equation*}
and
\begin{equation*}
  \begin{aligned}
   (1-x^2)  G_1''(x) = (-2(c_2-4c_3)+O(r^2))r^2.
  \end{aligned}
\end{equation*}
It follows from \eqref{G1''-upbd} that
\begin{equation}\label{G1-upbd}
  G_1\leq\frac{5!}{20\alpha}=\frac{3!}{\alpha},\quad\forall x\in[-1,1].
\end{equation}
\par
 Case 2: if
 \begin{equation*}
  M_1=\lim_{x_k\rightarrow -1}G_1'(x_k) \quad \mbox{ for  some } x_k\in(-1,1).
\end{equation*}
Then it is similar to the  case 1 to show \eqref{G1-upbd}.

Case 3:  Let $M_1=G_1(x_0)$ for some $x_0\in(-1,1)$.  Then
\begin{equation*}
   G_1'(x_0)=0 \quad\mbox{and }\quad   G''(x_0)\leq0.
\end{equation*}
 \eqref{G1-upbd}  immediately follows from \eqref{G1''-upbd}.
 \par
 We see from \eqref{G1-upbd} that
 \begin{equation*}
   \begin{aligned}
    ((1-x^2)G)^{(3)}&=-6G'-6xG''+(1-x^2)G'''\geq -\frac{6}{\alpha},\quad\forall x\in[-1,1].
   \end{aligned}
 \end{equation*}
Repeating the previous arguments, we can conclude
\begin{equation*}
  G'\leq \frac{1}{\alpha}, \quad\forall x\in[-1,1].
\end{equation*}

In general,  we can start with $G_{j}, j=\frac n2 -2$.

In view of \eqref{G-eq0},  one  directly has
\begin{equation*}
(-1)^\frac{n}{2}((1-x^2)^\frac{n-2}{2}G)^{(n-1)}\le \frac{(n-1)!}{\alpha}.
\end{equation*}

Since
\begin{equation*}
 (-1)^\frac{n}{2}((1-x^2)^\frac{n-2}{2}G)^{(n-1)}=-(1-x^2) G_{\frac n2-2}''+2(n-1)G_{\frac n2-2}'+(n-1)(n-2) G_{\frac n2-2},
\end{equation*}
we can follow the same arguments above to obtain \eqref{gradient-n} for $j=\frac n2 -2$.
Then we can apply mathematical induction on  $j$ from $\frac n2-2$ to $0 $ to  conclude \eqref{gradient-n} for all  $0\le j \le \frac n2-1.$
\end{proof}

\subsection{The Case:\ n=6}
\vskip 3mm
\noindent\par

 Throughout the rest of this subsection,
we will focus on $n=6$ and  the equation
    \begin{equation}\label{6-mfe}
    \alpha P_6 u+5!\left(1- \frac{e^{6u}}{\int_{\s^6}e^{6u}dw} \right)=0,
    \end{equation}
    where $P_6=-\Delta(-\Delta+4)(-\Delta+6)$.
  As previously introduced,
\begin{equation}\label{P6-axial}
  P_6u=-[(1-x^2)^3 u']^{(5)}=
  -[(1-x^2)^2\Delta u]^{(4)},
\end{equation}
and equation \eqref{6-mfe} is reduced to
  \begin{equation}\label{6-ode}
  -\alpha[(1-x^2)^3 u']^{(5)}+5!- \frac{2^7}{\gamma}e^{6u}=0.
 \end{equation}
 In view of \eqref{decomp_G^2}-\eqref{decomp_Gn-2}, we derive that
\begin{equation}\label{n=6decomp-G''1}
  \int_{-1}^1(1-x^2)^4 (G'')^2=\sum_{k=1}^\infty\bar\lambda_k(\bar\lambda_k-6) t_k^2.
\end{equation}
Moreover,
\eqref{decomp_Gn-2} can be rewritten as
\begin{equation}\label{n=6decomp-G''}
 \int_{-1}^1|[(1-x^2)^2 G]^{''}|^2 =
 \sum_{k=1}^\infty(\bar\lambda_k+4)(\bar\lambda_k+6) t_k^2.
\end{equation}

\begin{lemma}
Let $n=6$, using the semi-norm  $  \lfloor G\rfloor$  defined in \eqref{norm}, we have the following estimate:
\begin{equation}\label{n=6keyestimate}
  \lfloor G\rfloor^2\leq \left(\frac{30}{\alpha}-15\right)\int_{-1}^1 |[(1-x^2)^2G]''|^2-\frac{320}{\alpha}\int_{-1}^1 (1-x^2)^3(G')^2.
\end{equation}
\end{lemma}

\begin{proof}

By \eqref{decomp_G^2}, \eqref{n=6decomp-G''1}, \eqref{n=6decomp-G''} and Lemma  \ref{Gj-estimate-lemma}, we get
\begin{equation*}
  \begin{aligned}
  \lfloor G\rfloor^2 +15 \int_{-1}^1|[(1-x^2)^2 G]^{''}|^2&\le \frac{1}{\alpha}\int_{-1}^1 \left[6!(1-x^2)^2 G^2
      +30 (1-x^2)^4 (G'')^2+160(1-x^2)^3(G')^2\right]\\
   &=\frac{1}{\alpha}\int_{-1}^1 \left[30\left|[(1-x^2)^2 G]^{''}\right|^2-320(1-x^2)^3(G')^2\right].
  \end{aligned}
\end{equation*}
So \eqref{n=6keyestimate} holds.
   \end{proof}
\begin{proposition}\label{2/3}
If $\frac{2}{3}<\alpha<1$, any axially symmetric solution to \eqref{6-mfe} must be constant.
\end{proposition}
\begin{proof}
 We only need to show that
   either $\alpha\le \frac23$ or that $G$ and hence $u$ is identically $0$. Indeed,   when $\alpha>2/3$, it follows from \eqref{n=6keyestimate} that
 \begin{equation*}
  \begin{aligned}
  \lfloor G\rfloor^2 +15 \int_{-1}^1|[(1-x^2)^2 G]^{''}|^2&<\frac32\sum_{k=1}^\infty (30\bar\lambda_k^2-20\bar\lambda_k+720)t_k^2.
  \end{aligned}
\end{equation*}
Equivalently,
 \begin{equation*}
  \begin{aligned}
 0&>\sum_{k=1}^\infty \left[(\bar\lambda_k+15)(\bar\lambda_k+4)(\bar\lambda_k+6)-\frac32(30\bar\lambda_k^2
 -20\bar\lambda_k+720)\right]t_k^2\\
 &=\sum_{k=1}^\infty  (\bar\lambda_k-6) (\bar\lambda_k^2-14\bar\lambda_k+120) t_k^2.
  \end{aligned}
\end{equation*}
In view of \eqref{Ck-ode} the assertion holds.
\end{proof}

For $n=6$, \eqref{C1G} and \eqref{key-id} become
\begin{equation}\label{n=6C1G}
     \int_{-1}^1(1-x^2)^2 C_1 G=\frac{16}{105}\beta
\end{equation}
and
\begin{equation}\label{n=6key-id}
     \int_{-1}^1 |[(1-x^2)^2G]''|^2=\frac{256}{35}\left(7-\frac{1}{\alpha}\right)\beta.
\end{equation}

Inspired by   \cite{Gui00} and \cite{Gui98}, our basic strategy is to assume $\beta\neq 0$, and show that  it leads to a contradiction with the range of $\alpha$. It is fairly easy to see from \eqref{n=6key-id} that
\begin{equation}\label{contradiction}
   \mbox{ if }\beta=0,    \mbox{ then }\nabla u=0,   \mbox{   which shows that $u$ is a constant.}
\end{equation}
In what follows,  suppose that $\beta\neq0$ and
\begin{equation}\label{alpha-range1}
  \frac35<\alpha\le \frac23.
\end{equation}
  Then it is easy to see from \eqref{n=6C1G} and \eqref{enu} that
\begin{equation}\label{beta-range1}
  0<\beta<\frac{1}{\alpha}.
\end{equation}

\begin{proof}[Proof of Theorem \ref{even} for $n=6$]

We first define the following quantity
\begin{equation}\label{D}
  D:=\sum_{k=3}^\infty\left[\bar\lambda_k(\bar\lambda_k+4)(\bar\lambda_k+6)
  -\left(14+\frac{112}{9\alpha}\right)
  (\bar\lambda_k+4)(\bar\lambda_k+6)+\frac{320}{\alpha}\bar\lambda_k\right]t_k^2.
\end{equation}
We now give the upper bound of $D$. From \eqref{n=6keyestimate}   we see that
\begin{equation}\label{D-upbd}
\begin{aligned}
  D &=\lfloor G\rfloor^2 -\left(14+\frac{112}{9\alpha}\right)\int_{-1}^1|[(1-x^2)^2 G]^{''}|^2+\frac{320}{\alpha}
  \int_{-1}^1 (1-x^2)^3(G')^2\\
  &-(\frac{1280}{3\alpha}-960)\beta^2\int_{-1}^1(1-x^2)^2C_1^2\\
    &\leq \left(\frac{158}{9\alpha}-29\right)\int_{-1}^1|[(1-x^2)^2 G]^{''}|^2+\frac{16}{21}\left(192-\frac{256}{3\alpha}\right) \beta^2 \\
    &\leq \frac{16\beta}{7}\left[
    \frac{16}{5}\left(7-\frac{1}{\alpha}\right)\left(\frac{158}{9\alpha}-29\right)
    +\left(64-\frac{256}{9\alpha}\right)\beta\right].
   \end{aligned}
\end{equation}
It is easy to check $D>0.$ Combining \eqref{beta-range1} and \eqref{D-upbd}, we have
\begin{equation*}
   \frac{16}{5}\left(7-\frac{1}{\alpha}\right)\left(\frac{158}{9\alpha}-29\right)
    +\left(64-\frac{256}{9\alpha}\right)\frac{1}{\alpha}>0.
\end{equation*}
Therefore,
\begin{equation}\label{alpha-up1}
 \alpha<\bar\alpha:=\frac{221 + \sqrt{13345}}{522}<\frac{2}{3}.
\end{equation}

\par
To estimate the refined  lower bound of $D$, we define
\begin{equation*}
  g(t)=t- \left(14+\frac{112}{9\alpha}\right)+\frac{320}{\alpha} \frac{t}{(t+4)(t+6)}.
\end{equation*}
Then it is easy to see that
\begin{equation*}
  g'(t)>0 \mbox{ for }t\ge \bar\lambda_3(=24).
\end{equation*}
 Thus,
\begin{equation*}
  \begin{aligned}
 D&=\sum_{k=3}^\infty\left[\bar\lambda_k- \left(14+\frac{112}{9\alpha}\right)+\frac{320}{\alpha} \frac{\bar\lambda_k}{(\bar\lambda_k+4)(\bar\lambda_k+6)} \right](\bar\lambda_k+4)(\bar\lambda_k+6)t_k^2\\
 &\geq\left[\bar\lambda_3- \left(14+\frac{112}{9\alpha}\right)+\frac{320}{\alpha} \frac{\bar\lambda_3}{(\bar\lambda_3+4)(\bar\lambda_3+6)} \right]\sum_{k=3}^\infty(\bar\lambda_k+4)(\bar\lambda_k+6)t_k^2\\
 &=\left(10-\frac{208}{63\alpha}\right)\sum_{k=3}^\infty
 (\bar\lambda_k+4)(\bar\lambda_k+6)t_k^2.
   \end{aligned}
\end{equation*}

On the other hand,
we derive from \eqref{Ck'}, \eqref{Ck-ortho} and \eqref{CkG} that
 \begin{equation}\label{bk-upbd}
 \begin{aligned}
  t_k^2&=d_k^2\int_{-1}^1(1-x^2)^2C_k^2=
  \left(\int_{-1}^1(1-x^2)^2C_k^2\right)^{-1}\left[\frac{128}{\alpha\lambda_k}\int_{-1}^1 \frac{e^{6u}}{\gamma}(1-x^2)^3C_k'\right]^2\\
  &\leq \frac{(2k+5)(\bar\lambda_k+4)(\bar\lambda_k+6)}{128}\left[\frac{8}{\alpha  \lambda_k }\frac{\bar\lambda_k}{7}
   (1-\alpha\beta)\right]^2 \\
  &= \frac{128(2k+5)}{49(\bar\lambda_k+4)(\bar\lambda_k+6)}\left(\frac{1}{\alpha}-\beta\right)^2,\quad k\geq 2.
 \end{aligned}
\end{equation}
In particular,
 \begin{equation}\label{ak-upbd}
 \begin{aligned}
    d_k^2&\leq\left(\int_{-1}^1(1-x^2)^2C_k^2\right)^{-1}\frac{128(2k+5)}{
    49(\bar\lambda_k+4)(\bar\lambda_k+6)}
    \left(\frac{1}{\alpha}-\beta\right)^2\\
    &\le \left[\frac{2k+5}{7}\left(\frac{1}{\alpha}-\beta\right)\right]^2, \quad k\geq 2.
 \end{aligned}
\end{equation}
 By \eqref{n=6C1G}, \eqref{n=6key-id} and \eqref{bk-upbd}, we get the lower bound of $D$:
\begin{equation}\label{D-lowbd}
\begin{aligned}
 D&\geq\left(10-\frac{208}{63\alpha}\right)\sum_{k=3}^\infty
 (\bar\lambda_k+4)(\bar\lambda_k+6)t_k^2\\
    &= \left(10-\frac{208}{63\alpha}\right)\left[\int_{-1}^1|[(1-x^2)^2 G]^{''}|^2-\frac{128}{7}\beta^2-
    \frac{1152}{49}\left(\frac{1}{\alpha}-\beta\right)^2\right]\\
    &\geq \frac{16}{7} \left(10-\frac{208}{63\alpha}\right) \left[ \frac{16\beta}{5}\left(7-\frac{1}{\alpha}\right)-8\beta^2
    -\frac{72}{7}\left(\frac{1}{\alpha}-\beta\right)^2\right].
   \end{aligned}
\end{equation}
By both \eqref{D-upbd} and \eqref{D-lowbd}, we see that
\begin{equation*}
\begin{aligned}
   &\quad  \frac{16\beta}{7}\left[
    \frac{16}{5}\left(7-\frac{1}{\alpha}\right)\left(\frac{158}{9\alpha}-29\right)
    +\left(64-\frac{256}{9\alpha}\right)\beta\right] \\
    &\geq \frac{16}{7} \left(10-\frac{208}{63\alpha}\right) \left[ \frac{16\beta}{5}\left(7-\frac{1}{\alpha}\right)-8\beta^2
    -\frac{72}{7}\left(\frac{1}{\alpha}-\beta\right)^2\right].
   \end{aligned}
\end{equation*}
A straightforward computation shows that
\begin{equation*}
  \begin{aligned}
   0& \leq  \frac{16\beta}{5}\left(7-\frac{1}{\alpha}\right)
   \left[\frac{158}{9\alpha}-29-\left(10-\frac{208}{63\alpha}\right)
    \right] +\beta^2\left[64-\frac{256}{9\alpha}
     +8\left(10-\frac{208}{63\alpha}\right) \right]\\
     &+ \frac{72}{7} \left(10-\frac{208}{63\alpha}\right)
    \left(\frac{1}{\alpha}-\beta\right)^2\\
    &=\frac{16\beta}{5}\left(7-\frac{1}{\alpha}\right)
   \left(\frac{146}{7\alpha}-39\right)+48\beta^2\left(3-\frac{8}{7\alpha} \right)+ \frac{72}{7} \left(10-\frac{208}{63\alpha}\right)
   \left(\frac{1}{\alpha}-\beta\right)^2.
   \end{aligned}
\end{equation*}
We further have
\begin{equation}\label{ine-1}
\begin{aligned}
  &\quad \beta\left[\frac{2}{5}\left(7-\frac{1}{\alpha}\right)
   \left(\frac{146}{7\alpha}-39\right)+ \frac{1}{\alpha}\left(18-\frac{48}{7\alpha} \right)\right]\\
   &\ge\left(\frac{1}{\alpha}-\beta\right)  \left[\left(18-\frac{48}{7\alpha} \right)\beta- \left(\frac{90}{7}-\frac{208}{49\alpha}\right)
   \left(\frac{1}{\alpha}-\beta\right)\right]\\
   &:=\left(\frac{1}{\alpha}-\beta\right) I.
\end{aligned}
\end{equation}
 We  want to show that the  term $I$ is nonnegative.  Thus, we need to obtain the lower bound of $\beta.$
Note that the argument  exploited  in  \cite{Gui20}    is not applicable to this case. Precisely, following \cite{Gui20}, \eqref{D-upbd}  implies that
\begin{equation*}
 \left(64-\frac{256}{9\alpha}\right)\beta> \frac{16}{5}\left(7-\frac{1}{\alpha}\right)\left(29-\frac{158}{9\alpha}\right).
\end{equation*}
 However, the term $29-\frac{158}{9\alpha}$ is positive when $\alpha>\frac{158}{261}\approx0.6054$.
Hence, we   choose a suitable $\underline{\alpha}>\frac{158}{261}$ so  that
\begin{equation}\label{31}
 I>0,\quad \mbox{ for }  \alpha\in(\underline{\alpha},\bar\alpha).
\end{equation}
Then $\underline{\alpha}$ satisfies that
 \begin{equation*}
   \begin{aligned}
     I&\geq\left(18-\frac{48}{7\underline{\alpha}} \right)\beta+ \left(\frac{208}{49\bar\alpha}-\frac{90}{7} \right)
   \left(\frac{1}{\alpha}-\beta\right)\\
   &\geq\left(18+\frac{90}{7}-\frac{48}{7\underline{\alpha}}-\frac{208}{49\bar\alpha} \right)\beta+ \left(\frac{208}{49\bar\alpha}-\frac{90}{7} \right)
     \frac{1}{\underline{\alpha}}\\
     &\ge\left(\frac{216}{7}-\frac{48}{7\underline{\alpha}}
     -\frac{208}{49\bar\alpha} \right)
     \left(7-\frac{1}{\underline\alpha}\right)\left(29-\frac{158}{9\underline\alpha}\right)
     \frac{9\bar \alpha}{180\bar\alpha-80}+\left(\frac{208}{49\bar\alpha}-\frac{90}{7} \right)
     \frac{1}{\underline{\alpha}}\\
     &\geq0,
   \end{aligned}
 \end{equation*}
which indicates  $\underline{\alpha}\ge 0.61488$.   So we take $\underline{\alpha}=0.61488$.
 For $\alpha\in(\underline{\alpha},\bar\alpha)$, it follows from \eqref{ine-1} that
\begin{equation*}
  \frac{2}{5}\left(7-\frac{1}{\alpha}\right)
   \left(\frac{146}{7\alpha}-39\right)+ \frac{1}{\alpha}\left(18-\frac{48}{7\alpha} \right)>0.
\end{equation*}
Hence we find
\begin{equation}\label{alpha6}
   \alpha< \alpha^{(6)}:=\frac{115 + \sqrt{2851}}{273} \approx0.61683.
\end{equation}
Theorem \ref{even} for $n=6$ is proven.
\end{proof}

\vskip 3mm
\subsection{The Case:\ n=8}
\vskip 3mm
\noindent\par
This subsection focuses on the case $n=8$.   As we have addressed,
\begin{equation}\label{P8}
  P_8=-\Delta(-\Delta+6)(-\Delta+10)(-\Delta+12),
\end{equation}
and equation \eqref{n-mfe} for $n=8$  becomes
  \begin{equation}\label{8-ode}
  \alpha[(1-x^2)^4 u']^{(7)}+7!- \frac{9*2^9}{\gamma}e^{8u}=0.
 \end{equation}

 In view of \eqref{decomp_G^2}-\eqref{decomp_Gn-2}, we derive that
\begin{equation}\label{n=8decomp-G3-1}
\int_{-1}^1(1-x^2)^5(G'')^2=\sum_{k=1}^\infty\bar\lambda_k(\bar\lambda_k-8) t_k^2
\end{equation}
and
\begin{equation}\label{n=8decomp-G2}
  \int_{-1}^1(1-x^2)^6(G^{(3)})^2=\sum_{k=1}^\infty\bar\lambda_k(\bar\lambda_k^2-26\bar\lambda_k+144) t_k^2.
\end{equation}
Moreover,
\eqref{decomp_Gn-2} can be rewritten as
\begin{equation}\label{n=8decomp-G3}
 \int_{-1}^1|[(1-x^2)^3 G]^{(3)}|^2 =
 \sum_{k=1}^\infty \tilde\lambda_k t_k^2,
\end{equation}
where \begin{equation}\label{tildek}
 \tilde\lambda_k=(\bar\lambda_k+6)(\bar\lambda_k+10)
  (\bar\lambda_k+12).
\end{equation}
\vskip 2mm
\begin{lemma}
Let $n=8$, then we have the following estimate:
\begin{equation}\label{n=8keyestimate}
  \lfloor G\rfloor^2\leq 28\left(\frac{2}{\alpha}-1\right) \int_{-1}^1|[(1-x^2)^3 G]^{(3)}|^2-\frac{20160}{\alpha}\int_{-1}^1 (1-x^2)^4(G')^2,
\end{equation}
where $\lfloor G\rfloor$ is defined in \eqref{norm}.
\end{lemma}

\begin{proof}

By \eqref{n=8key-eq1}, Lemma  \ref{Gj-estimate-lemma} and \eqref{n=8decomp-G2}-\eqref{n=8decomp-G3}, we get
\begin{equation*}
  \begin{aligned}
  &\quad
  \lfloor G\rfloor^2 +28\int_{-1}^1|[(1-x^2)^3 G]^{(3)}|^2\\
  &\le \frac{56}{\alpha}\int_{-1}^1 \left[6!(1-x^2)^3 G^2+180
     (1-x^2)^4 (G')^2+54(1-x^2)^5(G'')^2+ (1-x^2)^6(G^{(3)})^2\right]\\
   &=\frac{56}{\alpha}\sum_{k=1}^\infty  \left[720+180\bar\lambda_k+54\bar\lambda_k(\bar\lambda_k-8)
   +\bar\lambda_k(\bar\lambda_k^2-26\bar\lambda_k+144) \right]t_k^2\\
   &=\frac{56}{\alpha}\sum_{k=1}^\infty \left[(\lambda_k+6)(\lambda_k+10)(\lambda_k+12)-360\bar\lambda_k  \right]t_k^2\\
   &=\frac{56}{\alpha}\int_{-1}^1 \left[|[(1-x^2)^3 G]^{(3)}|^2-360(1-x^2)^4(G')^2\right].
  \end{aligned}
\end{equation*}
The proof is complete.
   \end{proof}

   \begin{proposition}\label{21/25}
If $\frac{21}{25}\le\alpha<1$, any axially symmetric solution to \eqref{n-mfe}  must be constant.
\end{proposition}
\begin{proof}

  As in the proof of Proposition \ref{2/3}, when   $\alpha>\frac{21}{25}$, it follows from \eqref{n=8keyestimate} that
 \begin{equation*}
  \begin{aligned}
  \lfloor G\rfloor^2 +28\int_{-1}^1|[(1-x^2)^3 G]^{(3)}|^2&=\frac{56}{\alpha}\sum_{k=1}^\infty  \left[
  \bar\lambda_k^3+28\bar\lambda_k^2-108\bar\lambda_k+720 \right]t_k^2\\
  &<\frac{200}3\sum_{k=1}^\infty  \left[
  \bar\lambda_k^3+28\bar\lambda_k^2-108\bar\lambda_k+720 \right]t_k^2.
  \end{aligned}
\end{equation*}
Equivalently,
 \begin{equation*}
  \begin{aligned}
 0&>\sum_{k=1}^\infty \left[(\bar\lambda_k+28)\tilde\lambda_k- \frac{200}3(\bar\lambda_k^3+28\bar\lambda_k^2-108\bar\lambda_k+720 )\right]t_k^2\\
 &=\sum_{k=1}^\infty \left(\bar\lambda_k^4-\frac{32}{3}\bar\lambda_k^3-\frac{2492}{3}\bar\lambda_k^2-14976\bar\lambda_k
 -27840 \right) t_k^2>0,
  \end{aligned}
\end{equation*}
since $\bar\lambda_k\ge8$ and the equation
\begin{equation*}
  t^4-\frac{32}{3}t^3-\frac{2492}{3}t^2-14976t
 -27840=0
\end{equation*}
has two real solutions $t_1\approx-31.6$ or $t_2\approx2.1$.  This is a contradiciton, which completes the proof.
\end{proof}

\begin{remark}
  We note that it fails to get the same result  as that in Proposition \ref{2/3}. The reason is the following:  when $\alpha>2/3$, it follows from \eqref{n=8keyestimate} that
 \begin{equation*}
  \begin{aligned}
  \lfloor G\rfloor^2 +28\int_{-1}^1|[(1-x^2)^3 G]^{(3)}|^2&=\frac{56}{\alpha}\sum_{k=1}^\infty  \left[
  \bar\lambda_k^3+28\bar\lambda_k^2-108\bar\lambda_k+720 \right]t_k^2\\
  &<84\sum_{k=1}^\infty  \left[
  \bar\lambda_k^3+28\bar\lambda_k^2-108\bar\lambda_k+720 \right]t_k^2.
  \end{aligned}
\end{equation*}
Equivalently,
 \begin{equation*}
  \begin{aligned}
 0&>\sum_{k=1}^\infty \left[(\bar\lambda_k+28)(\bar\lambda_k+6)(\bar\lambda_k+10)(\bar\lambda_k+12)-84(\bar\lambda_k^3+28\bar\lambda_k^2-108\bar\lambda_k+720 )\right]t_k^2\\
 &=\sum_{k=1}^\infty  (\bar\lambda_k-8) (\bar\lambda_k^3-20\bar\lambda_k^2-1476\bar\lambda_k+5040)t_k^2.
  \end{aligned}
\end{equation*}
There is no contradiction, since the equation
\begin{equation*}
  t^3-20t^2-1476t+5040=0
\end{equation*}
has the following real solutions
\begin{equation*}
  t_1\approx-31.6,\quad t_2\approx 3.3\quad t_3\approx48.4.
\end{equation*}
\end{remark}

Recall that \eqref{C1G} and \eqref{key-id} for  $n=8$ are reduced to
\begin{equation}\label{n=8C1G}
     \int_{-1}^1(1-x^2)^3 C_1 G= \frac{32}{315}\beta
\end{equation}
and
\begin{equation}\label{n=8key-id}
     \int_{-1}^1 |[(1-x^2)^3G]^{(3)}|^2= \frac{1024}{7}\left(9-\frac{1}{\alpha}\right)\beta.
\end{equation}

As in the proof of Theorem \ref{even} for $n=6$,  in what follows, we  assume that
\begin{equation}\label{n=8a-range}
\beta\neq0,\quad  \frac23<\alpha< \frac{21}{25}
\end{equation}
and note
\begin{equation*}
   0<\beta<\frac{1}{\alpha}.
\end{equation*}

\begin{proof}[Proof of Theorem \ref{even} for $n=8$]

 Define
\begin{equation}\label{E}
  E:=\sum_{k=3}^\infty\left[\bar\lambda_k \tilde\lambda_k -\left(18+\frac{18}{\alpha}\right)
   \tilde\lambda_k
  +\frac{20160}{\alpha}\bar\lambda_k\right]t_k^2.
\end{equation}
We now  present  the upper bound of $E$. From \eqref{n=8keyestimate}-- \eqref{n=8key-id}  we derive
\begin{equation}\label{E-upbd}
\begin{aligned}
  E&=\lfloor G\rfloor^2 - \left(\frac{18}{\alpha}+18\right) \int_{-1}^1|[(1-x^2)^3 G]^{(3)}|^2+\frac{20160}{\alpha}\int_{-1}^1 (1-x^2)^4(G')^2\\
  &-7!\left(\frac{14}{\alpha}-10\right)\beta^2\int_{-1}^1(1-x^2)^2C_1^2\\
    &\leq \left(\frac{38}{\alpha}-46\right)\int_{-1}^1|[(1-x^2)^3 G]^{(3)}|^2-1024\left(\frac{7}{\alpha}-5\right)\beta^2 \\
    &=\frac{1024\beta}{7}\left[
   \left(\frac{38}{\alpha}-46\right)\left(9-\frac{1}{\alpha}\right)
    +7\left(5-\frac7{\alpha}\right)\beta\right].
   \end{aligned}
\end{equation}

To estimate the lower bound of $D$,
define
  \begin{equation*}
    \bar g(t)=t -\left(18+\frac{18}{\alpha}\right) +\frac{20160}{\alpha}\frac{t}{(t+6)(t+10)(t+12)}.
  \end{equation*}
 Differentiating $\bar g(t)$, we have
 \begin{equation*}
    \bar g'(t)=1-   \frac{20160}{\alpha} \frac{2 \left(t^3+14 t^2-360\right)}{(t+6)^2 (t+10)^2 (t+12)^2}.
 \end{equation*}
After some calculations, we deduce that $\bar g''(t)>0$ for $t>\bar\lambda_3(=30)$. Thus,  for $t\ge 30$,
\begin{equation*}
  \bar g'(t)\ge\bar g'(30)=1-\frac{109}{252 \alpha}>0
\end{equation*}
 due to \eqref{n=8a-range}. We further have
\begin{equation}\label{E-lowbd1}
\begin{aligned}
E&=\sum_{k=3}^\infty\left[\bar\lambda_k  -\left(18+\frac{18}{\alpha}\right)
  +\frac{20160\bar\lambda_k}{\alpha  \tilde\lambda_k}\right]\tilde\lambda_kt_k^2\\
  &\geq\left(12-\frac{8}{\alpha}\right)\sum_{k=3}^\infty \tilde\lambda_kt_k^2\\
  &>0,
   \end{aligned}
\end{equation}
since $\alpha>\frac23$.  We conclude from \eqref{n=8a-range}, \eqref{E-upbd} and \eqref{E-lowbd1} that
\begin{equation*}
 0<7\left(\frac7{\alpha}-5\right)\beta< \left(\frac{38}{\alpha}-46\right)\left(9-\frac{1}{\alpha}\right)
\end{equation*}
and so
\begin{equation}\label{alpha-n=8}
   \alpha< \alpha^{(8)}:=\frac{19}{23}<\frac{21}{25}.
\end{equation}
\end{proof}

\begin{remark}
 We wish to obtain for $n=8$  a similar inequality  as \eqref{ine-1} for $n=6$ so that the range of $\alpha $ can be more refined. However, the similar approach seems not working as shown below.

 As in \eqref{bk-upbd}, we find
\begin{equation}\label{tk-upbd}
 \begin{aligned}
    \tilde\lambda_k t_k^2 \le \frac{512(2k+7)}{9 }\left(\frac{1}{\alpha}-\beta\right) ^2, \quad k\geq 2.
 \end{aligned}
\end{equation}
 It follows from \eqref{n=8C1G}, \eqref{n=8key-id}    and \eqref{tk-upbd} that \begin{equation}\label{E-lowbd}
\begin{aligned}
E &\geq\left(12-\frac{8}{\alpha}\right)\sum_{k=3}^\infty \tilde\lambda_kt_k^2\\
&\ge\left(12-\frac{8}{\alpha}\right)
\left[\int_{-1}^1|[(1-x^2)^3 G]^{(3)}|^2-512\beta^2-
    \frac{512*11}{9}\left(\frac{1}{\alpha}-\beta\right)^2\right]\\
    &=512\left(12-\frac{8}{\alpha}\right)  \left[ \frac{2\beta}{7}\left(9-\frac{1}{\alpha}\right) -\beta^2
    -\frac{11}{9}\left(\frac{1}{\alpha}-\beta\right)^2\right].
   \end{aligned}
\end{equation}

Combining both \eqref{E-upbd} and \eqref{E-lowbd}, we  conclude that
\begin{equation*}
\begin{aligned}
   &\quad  \frac{2\beta}{7}\left[
   \left(\frac{38}{\alpha}-46\right)\left(9-\frac{1}{\alpha}\right)
    +7\left(5-\frac7{\alpha}\right)\beta\right] \\
    &\geq \left(12-\frac{8}{\alpha}\right)  \left[ \frac{2\beta}{7}\left(9-\frac{1}{\alpha}\right) -\beta^2
    -\frac{11}{9}\left(\frac{1}{\alpha}-\beta\right)^2\right].
   \end{aligned}
\end{equation*}
Equivalently,
\begin{equation*}
  \begin{aligned}
    0&\le \frac{2\beta}{7}\left(9-\frac{1}{\alpha}\right)
 \left(\frac{46}{\alpha}-58\right)
   +22\beta^2 \left(1-\frac{1}{\alpha}\right)
   +\frac{11}9\left(12-\frac{8}{\alpha}\right)
   \left(\frac{1}{\alpha}-\beta\right)^2.
   \end{aligned}
\end{equation*}
 Furthermore,
\begin{equation}\label{key-ine}
\begin{aligned}
  &\quad \beta\left[ \frac{2 }{7}\left(9-\frac{1}{\alpha}\right)
 \left(\frac{46}{\alpha}-58\right)
   +\frac{22}\alpha  \left(1-\frac{1}{\alpha}\right) \right]\\
   &\ge\left(\frac{1}{\alpha}-\beta\right)  \left[ 22\beta \left(1-\frac{1}{\alpha}\right)
   -\frac{11}9\left(12-\frac{8}{\alpha}\right)
   \left(\frac{1}{\alpha}-\beta\right)\right].
\end{aligned}
\end{equation}
However, the r.h.s. of \eqref{key-ine} is negative.  Thus, the previous argument for  $n=6$ is not applicable here.
\end{remark}

Theorem \ref{even} follows from \eqref{contradiction}, \eqref{alpha6} and \eqref{alpha-n=8}.

Next we shall show Theorem \ref{Szego}.

\begin{proof}[Proof of Theorem \ref{Szego}]

Following \cite{Chang95},  we define $\phi_{P, t}, P \in \mathbb{S}^n, t>0$ to be $ \phi_{P, t}(\xi)=\tilde \xi:= \pi_{P}^{-1}(ty)$, where $y=\pi_{P}(\xi)$ is the stereographic project of $\mathbb{S}^n$  from $P$ as the north pole to the equatorial plane.
In particular, we denote $\phi_{t}=\phi_{P_0, t}$ where $P_0=(1, 0, \cdots 0) $.

Given $u \in H^\frac n2(\mathbb{S}^n)$ and $t>0$,  let
$$
v(\xi)= u(\phi_t(\xi)) + \frac{n+1}{n^2} \ln |det(d \phi_t)|, \quad \xi \in \mathbb{S}^n.
$$

We first claim that  $\J_{\frac n{n+1}}$ owns the following invariance property:
\begin{equation}\label{invariance}
 \J_{\frac n{n+1}}(u)=\J_{\frac n{n+1}}(v), \quad  \forall u \in H^\frac n2(\mathbb{S}^n),  \, \, t>0.
\end{equation}
The  proof     has been carried out in detail for the case
 $n = 4$ in  \cite[Pro. 3.4]{Gui20}. The same argument there works for general $n$ with
slight modifications; so we will skip the proof here.  We further see that for any $  u \in H^\frac n2(\mathbb{S}^n)$,  there is a $\phi_{P, t}$ such that
$$
v(\xi)= u(\phi_{P,t}(\xi)) + \frac{n+1}{n^2}  \ln |det(d \phi_{P,t})|, \quad \xi \in \mathbb{S}^n
$$
belongs to $ \mathfrak L$.

In conclusion,  Theorem \ref{Szego} follows
immediately from Theorem \ref{even} and  \eqref{invariance}.
 \end{proof}

We note that  a similar but more general Sz\"ego limit theorem for $u\in H^1(\mathbb{S}^2)$ is proven in \cite{CG} using a  variational  method under a mass center constraint, in combination with the improved Moser-Trudinger inequality in \cite{GM}.

 \vskip4mm
{\section{ Bifurcation}}
\par
In this section we shall obtain results on bifurcation curves to \eqref{n-ode} in general for $\alpha>0$ and in particular for $\alpha  \in (\frac1{n+1}, \frac12)$.
We shall first apply  the standard bifurcation theory to analyze  the local bifurcation diagram.   Let us recall the following general theorem.
\par
\begin{theorem}{\rm(\cite[Theorem 1.7]{CR71})}\label{local}
 Let $X$,$Y$ be  Hilbert  spaces, $V$ a neighborhood of $0$ in $X$ and  $F:(-1,1)\times V\rightarrow Y$ a map with the following properties:
\begin{itemize}
\item[(1)]  $F(t,0)=0$ for any $t$;
\item[(2)]  $\partial_{t}F$,$\partial_{x}F$ and $\partial_{t,x}^{2}F$ exist and are continuous;
\item[(3)] $\ker(\partial_{x}F(0,0))=\mbox{span}\{w_{0}\}$ and $Y/\mathcal{R}(\partial_{x}F(0,0))$  are one-dimensional;
\item[(4)] $\partial_{t,x}^{2}F(0,0)w_0\not\in\mathcal{R}(\partial_{x}F(0,0))$.
\end{itemize}
If $Z$ is any complement of $ \ker(\partial_{x}F(0,0))$ in $X$.
Then there exists $\varepsilon_{0}>0$, a neighborhood  of $(0,0)$ in $U\subset(-1,1)\times X$ and continuously differentiable maps $\eta:(-\varepsilon_{0},\varepsilon_{0})\rightarrow\mathbb{R}$ and $z:(-\varepsilon_{0},\varepsilon_{0})\rightarrow Z$ such that $\eta(0) =0,\
z(0)=0$ and
\begin{equation}  \nonumber
F^{-1}(0)\cap U\setminus((-1,1)\times\{0\})=\{(\eta(\varepsilon),\varepsilon w_{0}+\varepsilon z(\varepsilon))\mid\varepsilon\in(-\varepsilon_{0},\varepsilon_{0})\}.\\
\end{equation}
\end{theorem}

Recall that the  shape of the above local  bifurcating branch can be further described  by the following theorem (see, e.g., \cite[I.6]{KH12}):
\begin{theorem}\label{local2}
In the setting of  Theorem \ref{local},  let $\psi\neq0\in Y^{-1}$  satisfy
\begin{equation*}
  \mathcal{R}(\partial_{x}F(0,0))=\{y\in Y\mid\langle\psi,y\rangle=0\},
\end{equation*}
   where $Y^{-1}$ is the dual space of $Y$.  Then we have
    \begin{equation*}\label{derivative}
 \eta'(0) =-\frac{\langle\partial^{2}_{x,x}F(0,0)[w_{0},w_{0}],
 \psi\rangle}{2\langle\partial^{2}_{t,x}F(0,0)w_{0},\psi\rangle}.
\end{equation*}
  Furthermore,  the  bifurcation is transcritical provided that   $ \eta'(0)\neq 0$.

\end{theorem}
\vskip 2mm
\par

Note that critical points of $I_\alpha(u)$
satisfy
\begin{equation}\label{41}
 \left(-1\right)^\frac{n}{2}[(1-x^2)^\frac{n}{2} u']^{(n-1)}+  \rho(1-x^2)^\frac{n-2}{2}\left(1-\frac{\sqrt\pi\Gamma\left(\frac{n}{2}\right)}{ \Gamma\left(\frac{n+1}{2}\right)}  \frac{ e^{nu}}{\int_{-1}^1(1-x^2)^\frac{n-2}{2}e^{nu}}\right) =0,
  \quad x \in (-1, 1),
\end{equation}
where $\rho=\frac{(n-1)!}{\alpha}$.

  Let
 \begin{equation*}
   \mathcal{V}=\bigg\{u\in H^n(\mathbb{S}^n):u=u(x),\ \int_{\mathbb{S}^n}udw=0 \bigg\}; \quad
   \mathcal{W}=\bigg\{u\in  L^2(\mathbb{S}^n):u=u(x),\ \int_{\mathbb{S}^n}udw=0\bigg\}
 \end{equation*}
and   define  a nonlinear operator $\T:\R\times  \mathcal{V}\rightarrow  \mathcal{W}$ as
\begin{equation*}
 \T(\rho,u)=P_n u+\rho \left(1- \frac{e^{nu}}{\int_{\mathbb{S}^n} e^{nu}dw}\right).
\end{equation*}
Obviously, the operator $\mathcal{T}$ is well defined.  After direct computations, one has
\begin{equation*}
    \partial_{u}\mathcal{T}(\rho,0)\phi=P_n\phi-n\rho\phi.
\end{equation*}

Define
\begin{equation*}
  \F(\rho,u)=u+\rho P_n^{-1} \left(1- \frac{e^{nu}}{\int_{\mathbb{S}^n} e^{nu}dw}\right).
\end{equation*}

 Let
 $\mathcal{S}$  denote the closure of the set of nontrivial solutions of
 \begin{equation}\label{47}
    \F(\rho,u)=0.
 \end{equation}
 It is clear that \eqref{47} and \eqref{41} are equivalent.

\par
Let $\lambda_k$  and $C_k^{\frac{n-1}{2}}$ be given in \eqref{PnCk}. Then by  similar arguments as in \cite[Lemma 5.3]{Gui20}, we have

\begin{theorem}\label{bifur1}
  Let $ \rho_k=\frac{\lambda_k}{n}$ for  $k=1,2,3,\dots$, the points $(\rho_{k},0)$ are  bifurcation points
 for the curve of solutions $(\rho,0)$. In particular,   there exists $\varepsilon_{0}>0$ and   continuously differentiable functions  $\rho_k:(-\varepsilon_{0},\varepsilon_{0})\rightarrow \R$ and $\psi_k:(-\varepsilon_{0},\varepsilon_{0})\rightarrow \{C_k^{\frac{n-1}{2}}\}^\bot$ such that $\rho_k(0)=\rho_k$, $\psi_k(0)=0$ and every
nontrivial solution of  \eqref{41} in a small neighborhood of $(\rho_{k},0)$ is of the form
\begin{equation*}
  (\rho_k(\varepsilon),\varepsilon C_k^{\frac{n-1}{2}}+\varepsilon \psi_k(\varepsilon)).
\end{equation*}
In particular, when $k=2$,  the bifurcation point $(\rho_2, 0)=(\frac{(n+1)!} {n}, 0) $ is a transcritical bifurcation point.  Indeed, we have
\begin{equation*}\label{trans}
\rho_2'(0)= - \frac{(n+1)!}{2} {\frac{ \int_{-1}^{1} (1-x^2)^{\frac{n-2}{2}}(C_k^{\frac{n-1}{2}})^3}{ \int_{-1}^{1} (1-x^2)^{\frac{n-2}{2}}(C_k^{\frac{n-1}{2}})^2}}= -\frac{(n+1)!(n-1)^2}{n(n+5)}\neq 0.
\end{equation*}
\end{theorem}

\begin{corollary}
  Let $ \alpha_k=\frac{n!}{\lambda_k}$ for  $k=1, 2,3,\dots$, the points $(\alpha_{k},0)$ are  bifurcation points for  the curve of solutions $(\alpha,0)$  of \eqref{n-ode}.  Moreover, when $k=2$,  the bifurcation point $(\frac1{n+1}, 0) $  is a  transcritical bifurcation point.
 \end{corollary}

\begin{remark}
When $k=1$,  the bifurcation leads to the  family of solutions $u=-\ln(1-ax),  a \in (-1, 1)$ and $\rho=(n-1)!$.
 It is clear that $(\rho_k, 0)$ is not a transcritical bifurcation point for $k $ odd since $C_k^{\frac{n-1}{2}}$ is an odd function and  $\rho'(0)=0$ in this case.  It should be true that $(\rho_k, 0)$ is a transcritical bifurcation point for $k $ even,  we only need to check if  {$ \int_{-1}^{1} (1-x^2)^{\frac{n-2}{2}}(C_k^{\frac{n-1}{2}})^3\not =0$} in this case, which can be confirmed for small $k$ numerically.   However,  in this paper we only need to use the transcriticality of $(\rho_2, 0)$.
 \end{remark}

In order to analyze  the global  bifurcation diagram,  we employ  a  global bifurcation theorem  via degree arguments (see \cite{KH12,R}) and also exploit  special properties of solutions to \eqref{41}.

First, we recall   a global  bifurcation result (see \cite[Theorem II.5.8]{KH12}).

 \begin{proposition}\label{t47}
  In Theorem \ref{bifur1},   the bifurcation at $(\rho_{k},0)$ is global
and  satisfies the Rabinowitz alternative, i.e., a global continuum of solutions to \eqref{41} either goes
to infinity  in $R \times \mathcal{W}$ or meets the trivial solution curve  at $(\rho_m, 0)$  for some $m \ge 1$ and $m\neq k$.
 \end{proposition}

 Next we state and prove  the following  more specific  global bifurcation result regarding \eqref{41}.

 \begin{theorem}\label{main-bifur}
 1)   For $k\ge 2$,  there exists a global continuum of solutions
 $\mathcal{B}^+_k  \subset \mathcal{S} \setminus \{ (\rho, 0), \rho \in \mathbb{R}\}$
  of \eqref{41}  which coincides  in a small neighborhood of $(\rho_k, 0)$ with
  $$\{ (\rho_k(\varepsilon),\varepsilon C_k^{\frac{n-1}{2}}+\varepsilon \psi_k(\varepsilon)),   \varepsilon<0\}.$$
    $\mathcal{B}^+_k $ is contained in
  $\mathcal{N}_2:= \{ (\rho, u):  \rho >  \frac{(n+1)!}n, \,  u \in L^2(-1, 1) \}$ and  is uniformly bounded in $L^2(-1, 1)$  for $\rho$ in
  any fixed  finite interval $[\rho_m, \rho_M] \subset (\frac{(n+1)!}n, \infty)$.  Furthermore, $ \mathcal{B}^+_k $  satisfies  the  improved Rabinowitz alternative, i.e.,  either $\mathcal{B}^+_k$ extends in $\rho$ to infinity  or meets the trivial solution curve at $(\rho_m, 0)$  for some $m \geq 2$.

 2)  Similarly, for $k\ge 2$,  there exists   a global continuum of solutions  $\mathcal{B}^{-}_k $  which coincides  in a small neighborhood of $(\rho_k, 0)$ with  $\{ (\rho_k(\varepsilon),\varepsilon C_k^{\frac{n-1}{2}}+\varepsilon \psi_k(\varepsilon)),   \varepsilon>0\} $.  When $k\ge 3$,  $\mathcal{B}^{-}_k $ is contained in
  $\mathcal{N}_2$  and satisfies  the  boundedness for  $\rho$ in any fixed  finite interval  $[\rho_m, \rho_M] \subset (2(n-1)!, \infty)$.  Furthermore,  the improved Rabinowitz alternative holds.

 3)  Moreover,    $\mathcal{B}^+_k=\{ u:  u(x)=v(-x),  v \in \mathcal{B}^{-}_k\} $  when $k$ is odd.

 4)  The global continuum of solutions $\mathcal{B}^{-}_2$ of \eqref{41} must be contained in the set
 $$\mathcal{N}_1:= \left\{ (\rho, u):  \rho \in \left( \frac{(n-1)! }{ \alpha^{(n)}}, \frac{(n+1)!}n\right) \supset (2(n-1)!, \frac{(n+1)!}n), \,  u \in L^2(-1, 1)\right \}.$$
    Furthermore,   $\mathcal{B}^{-}_2$ is unbounded in  $L^\infty ([-1, 1] )$,  and  there exists a sequence of $ (\rho^{(t)}, u^{(t)}) \in \mathcal{B}^{-}_2,  t =1, 2, \cdots $ such that $ \rho^{(t)} \to 2(n-1)!$ and
 $\|u^{(t)}\|_{L^\infty([-1,1])} \to \infty$.  As an immediate consequence, there is a nontrivial solution to \eqref{41} for any $\rho \in  (2(n-1)!, \frac{(n+1)!}n)$.
  \end{theorem}

 \begin{proof}
  The proof is similar to that of  the case $n=4$ in \cite{Gui20}. So we omit it.
 \end{proof}

\begin{proof}[Proof of Theorem \ref{nontrivial}]
 Theorem \ref{nontrivial} follows immediately from Theorem \ref{main-bifur}.  This leads to the existence of a nontrivial solution to \eqref{n-ode} for $\alpha \in (\frac1{n+1}, \frac12)$.
\end{proof}

\section{Appendix A: Proof of \eqref{n=6}}
\vskip 2mm
\noindent\par

In this appendix, we compute  the term $ \int_{-1}^1(1-x^2)^2G^2 [(1-x^2)^2G]^{(5)}$.

First,
\begin{equation}\label{3-1}
  \begin{aligned}
    &\quad \int_{-1}^1(1-x^2)^2G^2 [(1-x^2)^2G]^{(5)}= \int_{-1}^1(1-x^2)^4G^2 G^{(5)}-
    20\int_{-1}^1x(1-x^2)^3G^2 G^{(4)}\\
    &-40\int_{-1}^1(1-3x^2)(1-x^2)^2G^2 G^{(3)}+240\int_{-1}^1x(1-x^2)^2G^2 G^{''}
    +120\int_{-1}^1 (1-x^2)^2G^2 G^{'}\\
    &:=\sum_{i=1}^5 I_i.
      \end{aligned}
\end{equation}

Let
\begin{equation*}
\begin{aligned}
  G_5&=\int_{-1}^1(1-x^2)^4(G^3)^{(5)}=3\int_{-1}^1(1-x^2)^4G^2 G^{(5)}+30\int_{-1}^1(1-x^2)^4G G'G^{(4)}\\
     &+60\int_{-1}^1(1-x^2)^4 GG''G^{(3)}+60\int_{-1}^1(1-x^2)^4 (G')^2G^{(3)}+90\int_{-1}^1(1-x^2)^4 G'(G'')^2\\
&:=\sum_{i=1}^5 G_5i,
      \end{aligned}
\end{equation*}
\begin{equation*}
\begin{aligned}
   G_4&=\int_{-1}^1x(1-x^2)^3(G^3)^{(4)}=3\int_{-1}^1x(1-x^2)^3G^2 G^{(4)}+24\int_{-1}^1x(1-x^2)^3G G'G^{(3)}\\
     &+18\int_{-1}^1x(1-x^2)^3 G(G'')^2+36\int_{-1}^1x(1-x^2)^3 (G')^2G^{''}\\
&:=\sum_{i=1}^4 G_4i,
      \end{aligned}
\end{equation*}
$X_1=(1-7x^2)(1-x^2)^2$,  $X_2=(1-3x^2)(1-x^2)^2$   and
\begin{equation*}
  \begin{aligned}
  G_{3}^{(j)}&=\int_{-1}^1X_j(G^3)^{(3)}=3\int_{-1}^1X_j G^2 G^{(3)}+18\int_{-1}^1X_j GG'G''+
     6\int_{-1}^1X_j(G')^3 \\
&:=\sum_{i=1}^3 G_{3}^{(j)}i,\quad j=1,2.
      \end{aligned}
\end{equation*}
Here, we neglect the coefficients  before  $G_5i,\ G_4i$ and $G_3i$. For example,
$G_54=\int_{-1}^1(1-x^2)^4 (G')^2G^{(3)}.$

\par
After integration by parts,
\begin{equation}\label{G_5}
   \begin{aligned}
     & \int_{-1}^1(1-x^2)^4G^2 G^{(5)}=-\int_{-1}^1[(1-x^2)^4G^2]'G^{(4)}=8G_41-2G_52;\\
     & \int_{-1}^1(1-x^2)^4G G'G^{(4)}=8G_42-(G_53+G_54);\\
     &\int_{-1}^1(1-x^2)^4 GG''G^{(3)}=4G_43-\frac12 G_55;\\
     &\int_{-1}^1(1-x^2)^4 (G')^2G^{(3)}=8G_44-2G_55.
   \end{aligned}
\end{equation}
Similarly,
\begin{equation}\label{G_4}
   \begin{aligned}
     & \int_{-1}^1x(1-x^2)^3G^2 G^{(4)}=-G_{3}^{(1)}1-2G_42;\\
     &\int_{-1}^1x(1-x^2)^3G G'G^{(3)}=-G_{3}^{(1)}2-G_43-G_44\\
     &\int_{-1}^1x(1-x^2)^3 G(G'')^2=-G_{3}^{(1)}2-G_42-G_44\\
     &\int_{-1}^1x(1-x^2)^3 (G')^2G^{''}=-\frac13G_{3}^{(1)}3\\
   \end{aligned}
\end{equation}
Then
\begin{equation*}
\begin{aligned}
  &\quad I_1+I_2+a_1G_5+a_2G_4\\
  &=(1+3a_1)G_51+30a_1G_52+60a_1(G_53+G_54)+90a_1G_55\\
  &+(3a_2-20)G_41+24a_2G_42+18a_2G_43+36a_2G_44\\
  &=(24a_1-2)G_52+...+(24a_1+3a_2-12)G_41+...\\
  &=(24a_1-2)[8G_42-(G_53+G_54)]+60a_1(G_53+G_54)+90a_1G_55
  +(24a_1+3a_2-12)G_41+...\\
  &=(36a_1+2)(G_53+G_54)+90a_1G_55 +(24a_1+3a_2-12)G_41+(24a_2+192a_1-16)G_42+...\\
  &=-5G_55 +(24a_1+3a_2-12)G_41+(24a_2+192a_1-16)G_42+(144a_1+18a_2+8)(G_43+2G_44).\\
\end{aligned}
\end{equation*}
Thus,
\begin{equation*}
\begin{aligned}
  &\quad I_1+I_2+a_1G_5+a_2G_4+5G_55\\
  &=(24a_1+3a_2-12)G_41+(24a_2+192a_1-16)G_42+(144a_1+18a_2+8)(G_43+2G_44)\\
  &=-(24a_1+3a_2-12)G_{3}^{(1)}1+(144a_1+18a_2+8)(G_42+G_43+2G_44)\\
  &=-(24a_1+3a_2-12)G_{3}^{(1)}1-(144a_1+18a_2+8)G_{3}^{(1)}2
  +(144a_1+18a_2+8)G_44.
\end{aligned}
\end{equation*}
Let $t=24a_1+3a_2$, then we have
\begin{equation*}
\begin{aligned}
  &\quad I_1+I_2+a_1G_5+a_2G_4+5G_55+\frac{t-12}{3}G_{3}^{(1)}\\
 &=6(t-12)G_{3}^{(1)}2-(6t+8)G_{3}^{(1)}2+2(t-12)G_{3}^{(1)}3
  -\frac{6t+8}{3}G_{3}^{(1)}3\\
  &=-80G_{3}^{(1)}2
  -\frac{80}{3}G_{3}^{(1)}3.
\end{aligned}
\end{equation*}
Notice that
 \begin{equation}\label{G-3}
   \begin{aligned}
     &\int_{-1}^1X_j G^2 G^{(3)}=-\int_{-1}^1X_j'G^2G''-2\int_{-1}^1X_jGG'G''=-\int_{-1}^1X_j'G^2G''
     -2G^{(j)}_32\\
     &\int_{-1}^1X_j GG'G''=-\frac12\int_{-1}^1X_jG(G')^2-\frac12G^{(j)}_33. \\
   \end{aligned}
\end{equation}
Then
\begin{equation*}
\begin{aligned}
   &\quad I_3-80G_{3}^{(1)}2-\frac{80}{3}G_{3}^{(1)}3
   =40\int_{-1}^1X_2'G^2G''+80\int_{-1}^1[X_2-X_1]GG'G''-\frac{80}{3}G_{3}^{(1)}3\\
  &=40\int_{-1}^1X_2'G^2G''+320\int_{-1}^1x^2(1-x^2)^2GG'G''-\frac{80}{3}G_{3}^{(1)}3.
\end{aligned}
\end{equation*}
Thus,
\begin{equation*}
\begin{aligned}
   &\quad I_3+I_4-80G_{3}^{(1)}2-\frac{80}{3}G_{3}^{(1)}3
   =I_4+\int_{-1}^1 40X_2'G^2G''+320\int_{-1}^1x^2(1-x^2)^2GG'G''
   -\frac{80}{3}G_{3}^{(1)}3\\
   &=80\int_{-1}^1x(1-x^2)^2[3(1-x^2)-(9x^2-5)]G^2G''-320\int_{-1}^1x(1-x^2)(1-3x^2)G(G')^2\\
   &-\int_{-1}^1(1-x^2)^2\left[160x^2-\frac{80}{3}(1-7x^2)\right](G')^3\\
   &=160\int_{-1}^1x(1-x^2)(3x^2-1)(G^2G''+2G(G')^2)-\frac{80}{3}
   \int_{-1}^1(1-x^2)^3(G')^3\\
   &=\frac{160}{3}
   \int_{-1}^1x(1-x^2)(3x^2-1)(G^3)''-\frac{80}{3}
   \int_{-1}^1(1-x^2)^3(G')^3.
\end{aligned}
\end{equation*}
It is easy to see that
\begin{equation*}
\begin{aligned}
  \sum_{i=1}^5I_i&=-a_1\int_{-1}^1(1-x^2)^4(G^3)^{(5)}-a_2\int_{-1}^1x(1-x^2)^3(G^3)^{(4)}
  -\frac{t-12}{3}\int_{-1}^1
  (1-7x^2)(1-x^2)^2(G^3)^{(3)}\\
   &-\frac{160}{3}
   \int_{-1}^1x(1-x^2)(3x^2-1)(G^3)''+40\int_{-1}^1(1-x^2)^2(G^3)'\\
   &-5\int_{-1}^1(1-x^2)^4 G'(G'')^2-\frac{80}{3}
   \int_{-1}^1(1-x^2)^3(G')^3\\
   &=-5\int_{-1}^1(1-x^2)^4 G'(G'')^2-\frac{80}{3}
   \int_{-1}^1(1-x^2)^3(G')^3.
\end{aligned}
\end{equation*}

\section{Appendix B: Proof of \eqref{n=8}}

Compute $\int_{-1}^1(1-x^2)^3 G^2 [(1-x^2)^3 G]^{(7)}$. First,
\begin{equation}\label{8-1}
   \begin{aligned}
    &\quad\int_{-1}^1(1-x^2)^3 G^2 [(1-x^2)^3 G]^{(7)}= \int_{-1}^1(1-x^2)^6G^2 G^{(7)} -42\int_{-1}^1x(1-x^2)^5G^2 G^{(6)}\\
     &-126 \int_{-1}^1(1-x^2)^4\left(1-5 x^2\right) G^2 G^{(5)}+840 \int_{-1}^1x(1-x^2)^3\left(3-5 x^2\right) G^2 G^{(4)}\\
     &+2520\int_{-1}^1(1-x^2)^3\left(1-5 x^2\right) G^2 G^{(3)}-3\times 7!\int_{-1}^1x(1-x^2)^3 G^2 G''-7!\int_{-1}^1(1-x^2)^3 G^2 G'\\
    &:=\sum_{i=1}^7 I_i
      \end{aligned}
\end{equation}

Let
\begin{equation*}
\begin{aligned}
  G_7&=\int_{-1}^1(1-x^2)^6(G^3)^{(7)}=3\int_{-1}^1(1-x^2)^6G^2 G^{(7)}+42\int_{-1}^1(1-x^2)^6G G'G^{(6 )}\\
     &+126\int_{-1}^1(1-x^2)^6 GG''G^{(5)}+126\int_{-1}^1(1-x^2)^6 (G')^2G^{(5)}+210\int_{-1}^1(1-x^2)^6 GG^{(3)}G^{(4)}\\
     &+630\int_{-1}^1(1-x^2)^6 G'G''G^{(4)}+630\int_{-1}^1(1-x^2)^6 (G'')^2G^{(3)}+420\int_{-1}^1(1-x^2)^6 G'(G^{(3)})^2\\
&:=\sum_{i=1}^8G_7i;
      \end{aligned}
\end{equation*}
and
\begin{equation*}
\begin{aligned}
   G_6&=\int_{-1}^1x(1-x^2)^5(G^3)^{(4)}=3\int_{-1}^1x(1-x^2)^5G^2 G^{(6)}+36\int_{-1}^1x(1-x^2)^5G G'G^{(5)}\\
     &+90\int_{-1}^1x(1-x^2)^5 GG''G^{(4)}+90\int_{-1}^1x(1-x^2)^5 (G')^2G^{(4)}+60\int_{-1}^1x(1-x^2)^5 G(G^{(3)})^2\\
     &+360\int_{-1}^1x(1-x^2)^5 G'G''G^{(3)}+90\int_{-1}^1x(1-x^2)^5  (G'')^3\\
&:=\sum_{i=1}^7 G_6i.
      \end{aligned}
\end{equation*}
We consider these functions
\begin{equation}\label{X^ji}
\begin{aligned}
     X_5^{(1)}&=[x(1-x^2)^5]'=(1-x^2)^4(1-11x^2),\quad  X_5^{(2)}=(1-x^2)^4(1-5x^2);\\
     X_4^{(s)}&=\left[X_5^{(j)}\right]', \ s=1,2;\qquad  X_4^{(3)}=x(1-x^2)^3(3-5x^2);\\
     X_3^{(s)}&=\left[X_4^{(j)}\right]', \ s=1,2,3;\qquad  X_3^{(4)}=(1-x^2)^3(1-5x^2).
\end{aligned}
\end{equation}
Then  define
\begin{equation*}
  \begin{aligned}
  G_5^{(j)}&=\int_{-1}^1 X_5^{(j)} (G^3)^{(5)}=3\int_{-1}^1X_5^{(j)}G^2 G^{(5)}+30\int_{-1}^1X_5^{(j)}G G'G^{(4)}\\
     &+60\int_{-1}^1X_5^{(j)} GG''G^{(3)}+60\int_{-1}^1X_5^{(j)} (G')^2G^{(3)}+90\int_{-1}^1X_5^{(j)} G'(G'')^2\\
&:=\sum_{i=1}^5 G_5^{(j)}i,\quad j=1,2;
      \end{aligned}
\end{equation*}
\begin{equation*}
\begin{aligned}
   G_4^{(j)}&=\int_{-1}^1X_4^{(j)}(G^3)^{(4)}=3\int_{-1}^1X_4^{(j)}G^2 G^{(4)}+24\int_{-1}^1X_4^{(j)}G G'G^{(3)}\\
     &+18\int_{-1}^1X_4^{(j)} G(G'')^2+36\int_{-1}^1X_4^{(j)} (G')^2G^{''}\\
&:=\sum_{i=1}^4 G_4i, \quad j=1,2,3;
      \end{aligned}
\end{equation*}
  and
\begin{equation*}
  \begin{aligned}
 G_{3}^{(j)}&=\int_{-1}^1X_3^{(j)}(G^3)^{(3)}=3\int_{-1}^1X_3^{(j)} G^2 G^{(3)}+18\int_{-1}^1X_3^{(j)}GG'G''+
     6\int_{-1}^1X_3^{(j)}(G')^3\\
&:=\sum_{i=1}^3 G_{3}^{(j)}i, \quad j=1,2,3,4.
      \end{aligned}
\end{equation*}
As in Appendix A,   we neglect the coefficients  before  $G_7i,\ G_6i,\dots \ G_3^{(j)}i$.
\par
After integration by parts, one has
\begin{equation}\label{G-7}
   \begin{aligned}
     &G_71=\int_{-1}^1(1-x^2)^6G^2 G^{(7)}=-\int_{-1}^1[(1-x^2)^6G^2]'G^{(4)}=12G_61-2G_72;\\
     & G_72=12G_62-(G_73+G_74);\quad G_73=12G_63-(G_75+G_76);\\
     &G_74=12G_64-2 G_76;\quad G_75=6G_65-\frac12 G_78;\\
     &G_76=12G_66-(G_77+G_78);\quad G_77=4 G_67;
   \end{aligned}
\end{equation}
Similarly,
\begin{equation}\label{G-6}
   \begin{aligned}
     &-G_61=G_5^{(1)}1+2G_62;\quad &-G_62=G_5^{(1)}2+(G_63+G_64);\\
     &-G_63=G_5^{(1)}3+(G_65+G_66);  &-G_64=-G_5^{(1)}4+2G_66;\\
    &-G_65=G_5^{(1)}3+(G_63+G_66); &-G_66=\frac12G_5^{(1)}5+\frac12G_67;\\
   \end{aligned}
\end{equation}
\begin{equation}\label{G-5}
   \begin{aligned}
     &-G_5^{(j)}1=G_4^{(j)}1+2G_5^{(j)}2;\quad &-G_5^{(j)}2=G_4^{(j)}2+(G_5^{(j)}3+G_5^{(j)}4);\\
    &-G_5^{(j)}3=\frac12G_4^{(j)}3+\frac12G_5^{(j)}5;
    \quad &-G_5^{(j)}4=G_4^{(j)}4+2G_5^{(j)}5;\\
   \end{aligned}
\end{equation}
\begin{equation}\label{G-4}
   \begin{aligned}
     -G_4^{(j)}1=G_3^{(j)}1+2G_4^{(j)}2;\quad  -G_4^{(j)}2-G_4^{(j)}3=G_3^{(j)}2+G_4^{(j)}4;\quad
   -G_4^{(j)}4=\frac13G_3^{(j)}3.
   \end{aligned}
\end{equation}
By these relations, we calculate
\begin{equation*}
\begin{aligned}
  &\quad I_1+I_2+a_1G_7+a_2G_6\\
  &=(1+3a_1)G_71+42a_1G_72+126a_1(G_73+G_74)+210a_1G_75
  +630a_1(G_76+G_77)+420a_1G_78\\
  &+(3a_2-42)G_61+36a_2G_62+90a_2(G_63+G_64)+60a_2G_65
  +360a_2G_66+90a_2G_67 \\
  &=(36a_1-2)G_72+...+(36a_1+3a_2-30)G_61+...\\
  &=(90a_1+2)(G_73+G_74)+...(360a_1+30a_2+36)G_62+...-(36a_1+3a_2-30)G^{(1)}_51\\
  &=(120a_1-2)(G_75+3G_76)+630a_1G_77+420a_1G_78+(720a_1+60a_2-12)(G_63+G_64)\\
  &+60a_2G_65
  +360a_2G_66+90a_2G_67-(36a_1+3a_2-30)G^{(1)}_51-(360a_1+30a_2+36)G^{(1)}_52\\
  &=(270a_1+6)G_77+7G_78+3(720a_1+60a_2-12)G_66+90a_2G_67-(36a_1+3a_2-30)G^{(1)}_51\\
  &-(360a_1+30a_2+36)G^{(1)}_52-(720a_1+60a_2-12)(G^{(1)}_53+G^{(1)}_54)\\
  &=7G_78+42G_67-(36a_1+3a_2-30)G^{(1)}_51-...-(1080a_1+90a_2-18)G^{(1)}_55.
\end{aligned}
\end{equation*}
Let
\begin{equation*}
  II=I_1+I_2+a_1G_7+a_2G_6-7G_78-42G_67 \mbox{ and }12a_1+a_2=t.
\end{equation*}
Then it follows from  \eqref{G-5}  that
\begin{equation*}
  \begin{aligned}
    II&=-(3t-30)G^{(1)}_51
  -( 30t+36)G^{(1)}_52-(60t-12)(G^{(1)}_53+G^{(1)}_54)-(90t-18)G^{(1)}_55\\
  &=(3t-30)G^{(1)}_41-(24t+96)G^{(1)}_52-...\\
  &=(3t-30)G^{(1)}_41+(24t+96)G^{(1)}_42-(36t-108)(G^{(1)}_53+G^{(1)}_54)
  -(90t-18)G^{(1)}_55\\
  &=(3t-30)G^{(1)}_41+(24t+96)G^{(1)}_42+18(t-3)(G^{(1)}_43+G^{(1)}_44)
  -252G^{(1)}_55.
  \end{aligned}
\end{equation*}
 Joint with \eqref{G-4}, we further have
\begin{equation*}
  \begin{aligned}
    II+252G^{(1)}_55&=(3t-30)G^{(1)}_41+(24t+96)G^{(1)}_42+18(t-3)(G^{(1)}_43+G^{(1)}_44)\\
    &=-3(t-10)G^{(1)}_31+6(3t+26)G^{(1)}_42+18(t-3)(G^{(1)}_43+G^{(1)}_44)\\
    &=-3(t-10)G^{(1)}_31-6(3t+26)G^{(1)}_32-210G^{(1)}_43+6(3t-44)G^{(1)}_44\\
    &=-3(t-10)G^{(1)}_31-6(3t+26)G^{(1)}_32-2(3t-44)G^{(1)}_33-210G^{(1)}_43.
  \end{aligned}
\end{equation*}

Repeating the above arguments, we find
\begin{equation*}
\begin{aligned}
 -\frac1{126}I_3&= G^{(2)}_51=G^{(2)}_31-4G^{(2)}_32+2G^{(2)}_33-5G^{(2)}_43-5G^{(2)}_55;\\
 \frac1{840}I_4&=  G^{(3)}_41=-G^{(3)}_31+2G^{(3)}_32-\frac23G^{(3)}_33+2G^{(3)}_43.
\end{aligned}
\end{equation*}

Then we compute
\begin{equation*}
    \begin{aligned}
     &\quad  II+I_3+I_4+I_5+tG_3^{(1)}\\
     &=II-126G^{(2)}_51+840 G^{(3)}_41+2520 G^{(4)}_31+tG_3^{(1)}\\
      &=-252G^{(1)}_55+630G^{(2)}_55-210G^{(1)}_43+630G^{(2)}_43+1680G^{(3)}_43\\
      &+\left(30G^{(1)}_31 -126G^{(2)}_31-840G^{(3)}_31+2520G^{(4)}_31\right)
    +\left(1680G^{(3)}_32-156G^{(1)}_32+504G^{(2)}_32\right) \\
      &+
      \left(88G^{(1)}_33-252G^{(2)}_33 -1260G^{(3)}_33\right)\\
      &:=III_1+III_2+\sum_{i=1}^3III_3^{(i)}.
    \end{aligned}
\end{equation*}
By \eqref{X^ji}, we derive
\begin{equation*}
  \begin{aligned}
    III_1&=-252G^{(1)}_55+630G^{(2)}_55\\
    &=\int_{-1}^1\left[630 \left(1-5 x^2\right) \left(1-x^2\right)^4-252 \left(1-11 x^2\right) \left(1-x^2\right)^4\right]G'(G'')^2\\
    &=378\int_{-1}^1 \left(1-x^2\right)^5G'(G'')^2;
  \end{aligned}
\end{equation*}
\begin{equation*}
  \begin{aligned}
    III_2&=-210G^{(1)}_43+630G^{(2)}_43+1680G^{(3)}_43= 0;
  \end{aligned}
\end{equation*}
and
\begin{equation*}
  \begin{aligned}
    III_3^{(1)}&= 30G^{(1)}_31 -126G^{(2)}_31-840G^{(3)}_31+2520G^{(4)}_31=72\int_{-1}^1(25 x^4-48 x^2+19)(1-x^2)^2G^2 G^{(3)}\\
    &:=72\int_{-1}^1 X_3^{(5)}G^2 G^{(3)};\\
     III_3^{(2)}&=1680G^{(3)}_32-156G^{(1)}_32+504G^{(2)}_32=216\int_{-1}^1(15 x^4-26x^2+3)(1-x^2)^2GG'G''\\
     &:=216\int_{-1}^1X_3^{(6)}(1-x^2)^2GG'G'';\\
      III_3^{(3)}&=88G^{(1)}_33-252G^{(2)}_33 -1260G^{(3)}_33=72\int_{-1}^1(15 x^4-26x^2+3)(1-x^2)^2(G')^3\\
     &:=72\int_{-1}^1X_3^{(6)}(1-x^2)^2(G')^3.\\
  \end{aligned}
\end{equation*}
 It follows from \eqref{G-3} that
\begin{equation*}
  \begin{aligned}
    III_3&:=\sum_{i=1}^3III_3^{(i)}=-72\int_{-1}^1\left(X_3^{(5)}\right)'G^2 G''+\int_{-1}^1 \left(216X_3^{(5)}-144X_3^{(6)}\right)GG'G''\\
&+72\int_{-1}^1X_3^{(6)}(1-x^2)^2(G')^3\\
&=-72\int_{-1}^1\left(X_3^{(5)}\right)'G^2G''
-288\int_{-1}^1X_2G(G')^2+1260\int_{-1}^1(1-x^2)^4(G')^3,
  \end{aligned}
\end{equation*}
where $X_2=x(1-x^2)(15 5x^4-16x^2+19)$.
Then we consider
\begin{equation*}
  \begin{aligned}
      III_3+I_6+I_7&=-144\int_{-1}^1X_2\left[G^2G''+2G(G')^2\right]
      -5040\int_{-1}^1(1-x^2)^3G^2G'+1260\int_{-1}^1(1-x^2)^4(G')^3\\
      &=-48\int_{-1}^1X_2(G^3)'-1680\int_{-1}^1(1-x^2)^3(G^3)'
      +1260\int_{-1}^1(1-x^2)^4(G')^3.
  \end{aligned}
\end{equation*}

Put these results together,  we conclude that
\begin{equation*}
  \begin{aligned}
      \sum_{i=1}^7I_i&=-a_1\int_{-1}^1(1-x^2)^6(G^3)^{(7)}
      -a_2\int_{-1}^1x(1-x^2)^5(G^3)^{(6)}
  -t\int_{-1}^1
  X_3^{(1)}(G^3)^{(3)}\\
  &-48\int_{-1}^1X_2(G^3)'-1680\int_{-1}^1(1-x^2)^3(G^3)'+7\int_{-1}^1(1-x^2)^6 G'(G^{(3)})^2+42\int_{-1}^1x(1-x^2)^5  (G'')^3\\
  &+378\int_{-1}^1 \left(1-x^2\right)^5G'(G'')^2+1260\int_{-1}^1(1-x^2)^4(G')^3\\
  &=7\int_{-1}^1(1-x^2)^6 G'(G^{(3)})^2+42\int_{-1}^1x(1-x^2)^5  (G'')^3+378\int_{-1}^1 \left(1-x^2\right)^5G'(G'')^2\\
  &+1260\int_{-1}^1(1-x^2)^4(G')^3.\\
  \end{aligned}
\end{equation*}
Note that
\begin{equation*}
\begin{aligned}
  &\quad  \int_{-1}^1[(1-x^2)G]^{(3)}(1-x^2)^5 (G'')^2=\int_{-1}^1(1-x^2)^6 (G'')^2G^{(3)}\\
  &-6\int_{-1}^1x(1-x^2)^5  (G'')^3-6\int_{-1}^1(1-x^2)^5 G'(G'')^2\\
  &=-2\int_{-1}^1x(1-x^2)^5  (G'')^3-6\int_{-1}^1(1-x^2)^5 G'(G'')^2
\end{aligned}
\end{equation*}
Therefore,
\begin{equation*}
  \begin{aligned}
   &\quad\int_{-1}^1(1-x^2)^3 G^2 [(1-x^2)^3 G]^{(7)}=1260
   \int_{-1}^1(1-x^2)^4(G')^3+252\int_{-1}^1(1-x^2)^5 G'(G'')^2\\
   &+7\int_{-1}^1(1-x^2)^6G'(G^{(3)})^2 -21\int_{-1}^1[(1-x^2)G]^{(3)}(1-x^2)^5 (G'')^2.
\end{aligned}
\end{equation*}

\vskip4mm
$\mathbf{Acknowledgement}$
This research is partially supported by  NSF research grants DMS-1601885 and DMS-1901914.

\end{document}